\documentclass[a4paper]{amsart}

\newtheorem{theorem}{Theorem}[section]
\newtheorem{lemma}[theorem]{Lemma}

\newtheorem{proposition}[theorem]{Proposition}

\theoremstyle{definition}
\newtheorem{definition}[theorem]{Definition}

\newtheorem{assumption}[theorem]{Assumption}

\theoremstyle{remark}
\newtheorem{remark}[theorem]{Remark}

\usepackage{amsmath}
\usepackage{amsfonts}
\usepackage{amssymb}
\usepackage{upgreek}
\usepackage{mathtools}
\usepackage{hyperref}
\usepackage{mathrsfs}
\usepackage{stmaryrd}
\usepackage[inline]{enumitem}

\DeclarePairedDelimiter{\norm}{\|}{\|}

\newcommand*\diff{\mathop{}\!\mathrm{d}}

\newcommand{\ndg}[1]{| \kern -.25mm \|{#1}| \kern -.25mm \|}

\title[{\emph{A posteriori} error estimates for the Allen-Cahn problem}]{\emph{A posteriori} error estimates \\ for the Allen-Cahn problem}

\author{Konstantinos Chrysafinos}
\address{(K.~Chrysafinos) 1) Department of Mathematics, School of Mathematical and Physical Sciences, National Technical University of Athens, Zografou 15780, Greece and 2) IACM, FORTH, 20013 Heraklion, Crete, Greece.}
	\email{chrysafinos@math.ntua.gr}
\author{Emmanuil H.~Georgoulis}\address{(E.~H.~Georgoulis) 1) Department of Mathematics, University of Leicester, Leicester LE1 7RH, UK, 2) Department of Mathematics, School of Mathematical and Physical Sciences, National Technical University of Athens, Zografou 15780, Greece and 3) IACM, FORTH, 20013 Heraklion, Crete, Greece.}
  \email{Emmanuil.Georgoulis@le.ac.uk}
\author{Dimitra Plaka}\address{(D.~Plaka) Department of Mathematics, School of Mathematical and Physical Sciences, National Technical University of Athens, Zografou 15780, Greece.}
  \email{dplaka@central.ntua.gr}

\usepackage{amsopn}

\begin{document}

\maketitle

\begin{abstract}
 This work is concerned with the proof of \emph{a posteriori} error estimates for fully-discrete Galerkin approximations of the Allen-Cahn equation in two and three spatial dimensions. The numerical method comprises of the backward Euler method combined with conforming finite elements in space. For this method, we prove conditional type \emph{a posteriori} error estimates in the $L^{}_4(0,T;L^{}_4(\Omega))$-norm that depend polynomially upon the inverse of the interface length $\epsilon$. The derivation relies crucially on the availability of a spectral estimate for the linearized Allen-Cahn operator about the approximating solution in conjunction with a continuation argument and a variant of the elliptic reconstruction. The new analysis also appears to improve variants of known \emph{a posteriori} error bounds in $L_2(H^1)$, $L_\infty^{}(L_2^{})$-norms in certain regimes.
\end{abstract}



\section{Introduction}
The Allen-Cahn problem comprises of a singularly perturbed parabolic semi-linear parabolic partial differential equation (PDE) together with suitable initial and boundary conditions, viz.,
\begin{equation}\label{1.1}
\begin{aligned}
   u_t^{} - \Delta u + \frac{1}{{\epsilon}^2}(u^3 -u) & = f  & \text{in} & \ \Omega \times (0,T], \\ 
   u &= 0   & \text{on} & \ \partial\Omega \times (0,T], \\
   u(\cdot,0)&= u_0^{}   & \text{in} & \ \Omega ;
\end{aligned}
\end{equation}
  we assume that $\Omega\subset \mathbb{R}^d$ is a convex, polygonal ($d=2$) or polyhedral ($d=3$) domain of the Euclidean space $\mathbb{R}^d$, $T \in \mathbb{R}^{+} $, $0 < \epsilon \ll 1$, for sufficiently smooth initial condition $u_0^{}$ and forcing function $f$ (precise regularity statements will be given below). 

The  problem \eqref{1.1} belongs to the class of the so-called phase field PDEs models for solidification of a pure material, originally introduced by Allen \& Cahn \cite{Allen-Cahn1979} to describe  the phase separation process of a binary alloy at a fixed temperature. The nonlinear function $F(u):=u^3 -u$ is the derivative of the classical \emph{double-well potential} $\int F(u){\rm d} u$.
Due to the nature of the non-linearity, the  solution $u$ develops time-dependent interfaces $\Gamma^{}_t : = \lbrace x \in \Omega : u(x,t) = 0 \rbrace$, separating regions for which $u\approx  1$ from regions where $u\approx -1$. The solution moves from one region to another within the, so-called, \textit{diffuse interfaces} of length $O(\epsilon)$. For a recent comprehensive review of phase field models and their relationship to geometric flows, we refer to \cite{Du-Feng2019}. 

Realistically, $\epsilon$ should be orders of magnitude smaller than the physical domain of simulation. Therefore, the accurate and efficient numerical solution of such phase field models requires the resolution of the dynamic diffuse interfaces. This means that the discretisation parameters  of any numerical method used should provide sufficient numerical resolution to approximate the interface evolution accurately.  In the context of finite element methods, this is typically achieved via the use of very fine meshes in the vicinity of the interface region. In an effort to simulate at a tractable computational cost, especially for $d=3$, it is essential to design adaptive algorithms which are able to dynamically modify the local mesh size.

A standard error analysis of finite element approximations of \eqref{1.1} leads to \emph{a priori} estimates with unfavourable exponential dependence on $\epsilon^{-1}$. This is impractical even for moderately small interface length $\epsilon$. The celebrated works \cite{Chen1994,Mottoni-Schatzman1995,Alikakos-Fusco1993}  showed that uniform  bounds for the principal eigenvalue of the linearized Allen-Cahn spatial operator about the solution $u$
are possible as long as the evolving interface is smooth (cf., \eqref{spectral_classical} below). Such spectral estimates are used in the seminal work \cite{Feng-Prohl2003} whereby \emph{a priori} error bounds with only polynomial dependence on $\epsilon^{-1}$ for finite element methods have been proven, enabling also the proof of convergence to the sharp-interface limit. Moreover, assuming the validity of a spectral estimate about the exact solution $u$, allowed the proof of the first conditional-type \emph{a posteriori} error bounds for finite element methods approximating the Allen-Cahn problem in $L_2(H^1)$-norm, for which the condition depends only polynomially on $\epsilon$; this was presented in the influential works 
\cite{Kessler-Nochetto-Schmidt2004,Feng-Wu2005}. 

This direction of research has taken a further leap forward with the seminal works \cite{Bartels2005,Bartels-Muller2011,Bartels-Muller-Ortner2011}, whereby the principle eigenvalue of the linearized spatial Allen-Cahn operator \emph{about the numerical solution} $U_h$ is used instead, in an effort to arrive to fully computable \emph{a posteriori} error estimates in $L_2(H^1)$- and $L_{\infty}(L_2)$-norms, the latter using the elliptic reconstruction framework \cite{Makridakis-Nochetto2003,Lakkis-Makridakis2006}. 
We also mention \cite{Georgoulis-Makridakis2013} whereby \emph{a posteriori} error bounds in the $L^{}_{\infty}(L_r)$-norms, $r\in [2,\infty]$ are proven.

When the interface $\Gamma^{}_t$ undergoes topological changes, however, e.g., when an interface collapses, unbounded velocities occur and the all-important principal eigenvalue $\lambda$ can scale like $\lambda \sim \epsilon^{-2}$ on a time interval of length comparable to $\epsilon^2$. This crucial observation, made in \cite{Bartels-Muller-Ortner2011}, showed that the principal eigenvalue can be assumed to be $L_1$-integrable with respect to the time variable allowing, in turn, for robust conditional \emph{a posteriori} error analysis under topological changes in $L_2(H^1)$- and $L_{\infty}(L_2)$-norms.

 In a recent work \cite{Chrysafinos2019}, \emph{a priori} bounds for the $L^{}_4(L_4)$-norm error have been proved, which appear to deliver a rather favourable $\epsilon^{-1}$-polynomial dependence on the respective constant, noting that $L_4(L_4)$-norm is present in the stability of the spatial Allen-Cahn operator upon multiplication of \eqref{1.1} by $u$ and integration with respect to space and to time. An immediate question is whether proving conditional \emph{a posteriori} error bounds in $L_4(L_4)$-norm norm can also improve the dependence of the condition on the interface length $\epsilon$.
Motivated by this, in this work, we prove conditional \emph{a posteriori} error bounds for the $L_4(L_4)$-norm for a backward Euler-finite element method. The proof is valid under the hypothesis of the existence of a spectral estimate under topological changes in the spirit of \cite{Bartels-Muller-Ortner2011}. The argument uses a carefully constructed test function,  in conjunction with a continuation argument and a new variant of the elliptic reconstruction introduced in \cite{georgoulis_lakkis_wihler}. As a result of the method of proof, the new \emph{a posteriori} error analysis provides also new $L_{\infty}(L_2)$- and $L_2(H^1)$-norm \emph{a posteriori} error bounds which appear to, at least formally, be valid under less stringent smallness condition compared to results from the literature.

The remainder of this work is structured as follows. The model problem is introduced In Section \ref{sec:model}. Section \ref{sec:recon} include the definition of the numerical method along with the elliptic and time reconstructions needed for the proof of the main results. The key estimates and the main result are stated and proven in Section \ref{sec:apost}. Section \ref{sec:fully_computable} completes the derivation of fully computable error bounds by estimating the terms appearing in the residuals of the main results.


\section{Model problem}\label{sec:model}


We denote by $L_{p}^{}(\Omega)$, $1 \leq p \leq \infty$ the standard Lebesgue spaces with corresponding norms $\norm{ \cdot }^{}_{L_{p}^{}(\Omega)}$. Let also  $W^{k,p}(\Omega)$ is the $k$th order of Sobolev space based on $L_{p}^{}(\Omega)$ and $H^{k}(\Omega):= W^{k,2}(\Omega)$, $k \geq 0$, along with the corresponding norms $\norm{ \cdot }^{}_{W^{k,p}(\Omega)}$ and $\norm{ \cdot}^{}_{H^{k}(\Omega)}$, respectively. Set $H_{0}^{1}(\Omega) := \lbrace v \in H_{}^{1}(\Omega) : v \vert^{}_{\partial \Omega} = 0 \rbrace$. We shall denote by $\langle \cdot, \cdot \rangle$ the duality pairing between $H^{-1}(\Omega)$ and $H^{1}_{0}(\Omega)$, which becomes the standard $L^{}_2(\Omega)$ inner product $\left( \cdot, \cdot \right)$ when the arguments are sufficiently smooth.  The respective Bochner spaces are denoted by $L^{}_{p}(0,T;V)$, endowed with the norms:
\begin{align*}
    \norm{ v }^{}_{ L^{}_{p}(0,T;V) } = \Big( \int_{0}^{T} \norm{v}^{p}_{V} \diff t \Big)^{1/p},\ p\in[1,+\infty),\quad  \norm{ v }^{}_{ L^{}_{\infty}(0,T;V) } = \operatorname*{ess.sup}_{t\in[0,T]} \norm{v}^{}_{V},
\end{align*}
with $V$ a Banach space with norm $\norm{\cdot}_{V}$. 

We shall make extensive use of the classical \textit{Gagliardo-Nirenberg-Ladyzhenskaya inequalities (GNL)} reading: 
\begin{align}\label{eq:GNL}
\norm{ v}^{}_{L^{}_4(\Omega)} & \leq \tilde{c} \norm{ v}^{1/2} _{L^{}_2(\Omega)} \norm{ \nabla v}^{1/2} _{L^{}_2(\Omega)}, \ \text{for } d=2, \\
\norm{ v}^{}_{L^{}_4(\Omega)} & \leq \tilde{c} \norm{ v}^{1/4} _{L^{}_2(\Omega)} \norm{ \nabla v}^{3/4} _{L^{}_2(\Omega)}, \  \text{for }d=3,
\end{align}
for all $v \in H^{1}_{0}(\Omega)$ with $\tilde{c}>0$, independent of $v$. For later use, we also recall a basic algebraic estimate, often referred to as the \textit{Young's inequality}: for any  $\delta > 0 $, we have
\begin{align*}
a b \leq \delta a^p + C(p,q) \delta^{-\frac{q}{p}} b^{q}, \quad  \text{where} \quad  1/p  +  1/ q  = 1,
\end{align*}
for any $a,b \geq 0$  and $p,q > 1$, for some $C(p,q)>0$ independent of $a,b$. 

Let $f \in L^{}_\infty(0,T; L^{}_4(\Omega))$ and $u^{}_0 \in L^{}_{\infty}(\Omega)$. Then, for a.e. $t \in (0,T]$, we seek $u \in L^{}_2(0,T; H_0^1(\Omega)) \cap H^1 ( 0,T; H^{-1}(\Omega)) $, such that
\begin{equation}\label{2.1}
\begin{aligned}
    \langle u_t^{}(t),v \rangle + \left( \nabla u(t), \nabla v \right) + \epsilon^{-2} \left( u^3(t) - u(t),v \right) & = \langle f(t), v \rangle ,\\
    ( u(0),v )  & = ( u_0^{},v),
\end{aligned}
\end{equation}
for all $v \in H_0^1(\Omega).$ Integrating for $t\in(0,T]$, and integrating by parts the above becomes: find  $u \in L^{}_2(0,T; H_0^1(\Omega)) \cap L^{}_{\infty} ( 0,T; L^{}_2(\Omega)) $, such that
\begin{equation}\label{2.2}
\begin{aligned}
    & ( u(T),v(T) ) + \int_{0}^{T} \Big( - \langle u,v^{}_t \rangle  +( \nabla u, \nabla v )  + \epsilon^{-2} ( u^3 - u,v ) \Big) \diff t \\
    =&\ (u^{}_0,v(0) ) + \int_{0}^{T} \langle f,v \rangle \diff t,
\end{aligned}
\end{equation}
for all $v \in L^{}_2(0,T; H_0^1(\Omega)) \cap H^1 ( 0,T; H^{-1}(\Omega)) $.

\section{The fully discrete scheme and reconstructions}\label{sec:recon} We shall first present  a fully discrete scheme for the Allen-Cahn problem \eqref{1.1} by combining the lowest order discontinuous Galerkin time-stepping  method with conforming finite elements in space. Further, we shall define suitable space and time reconstructions of the fully discrete scheme, which will be crucial for the proof of the \emph{a posteriori} error bounds below.


\subsection{Discretisation} 
Let $0=t^{}_0 < t^{}_1 < \ldots < t^{}_N =T$. We partition the time interval $[0,T]$ into subintervals $J^{}_n := ( t^{}_{n-1}, t^{}_n ]$ and we denote by $k^{}_n := t^{}_n - t^{}_{n-1}$, $n=1,\ldots,N$ each time step.



Let also $\{\mathcal{T}^{n}_h\}_{n=0}^N$ be a sequence of conforming and shape-regular triangulations of the domain $\Omega$, that are allowed to be modified between time steps. 
We define the meshsize function, $h^{}_n : \Omega \rightarrow \mathbb{R}$, by $
h^{}_n(x):= \text{diam}(\tau)$,  $x \in \tau$ for  $\tau \in   \mathcal{T}^{n}_h$.
To each $\mathcal{T}^n_h$ we associate the finite element space:
$$V^{n}_h := \lbrace \chi \in C(\bar{\Omega}) ; \ \chi \vert^{}_{\tau} \in \mathbb{P}^{}_{\kappa}(\tau) , \ \forall \tau \in \mathcal{T}^{}_h \rbrace,$$
with $\mathbb{P}^{}_{\kappa}$ denoting the $d$-variate space of polynomials of degree at most $\kappa \in \mathbb{N}$. The whole theory presented below remains valid if box-type elements are used and respective polynomial spaces of degree $\kappa$ on each variable. 

We say that a set of triangulations is \textit{compatible} when they are constructed by different refinements of the same (coarser) triangulation. 
Given two compatible triangulations $\mathcal{T}^{n-1}_{h}$ and $\mathcal{T}^{n}_{h}$,
we consider their \textit{finest common coarsening}  $\hat{\mathcal{T}}^{n}_h:= \mathcal{T}^{n}_h \wedge \mathcal{T}^{n-1}_{h}$ and set $\hat{h}^{}_n := \max(h^{}_n,h^{}_{n-1})$. The partial order relation among the triangulations implies that $V^{n-1}_{h} \subset V^{n}_{h}$. Furthermore, we  denote by $\mathcal{S}^{n}_h$ the interior mesh skeleton of $\mathcal{T}^n_h$, and we define the sets $\hat{\mathcal{S}}^{n}_h := \mathcal{S}^{n}_h \cap \mathcal{S}^{n-1}_{h}$
and $\check{\mathcal{S}}^{n}_h := \mathcal{S}^{n}_h \cup \mathcal{S}^{n-1}_{h}.$


Approximations will be subordinate to the time partition. A finite element space $V^{n}_{h} \subset H^{1}_0(\Omega)$ is specified on each time interval $J^{}_n$, $n=1,\ldots,N$. Then, we seek approximate solutions from the space
\begin{align*}
   V^{n}_{hk} := \left\{  X :[0,T] \rightarrow V^{n}_h ; X \in L^{}_2(0,T;H^1_0(\Omega)) ; \ X \vert^{}_{J^{}_n} \in \mathbb{P}^{}_0\big[ J^{}_n ; V^{n}_{h} \big] \right\},
\end{align*}
with $\mathbb{P}^{}_0\big[ J^{}_n ; V^{}_{h} \big]$ denoting the space of constant polynomials over $J_n^{}$, having values in $V^n_{h}$; these functions are allowed to be discontinuous at the nodal points, but are taken to be continuous from the left.


\subsection{Fully discrete scheme} For brevity, we set $F(v) := v^3 -v $. The backward Euler-finite element method reads: for each $n=1,\ldots,N$, find $U^{n}_h \in V^n_{hk}$,  such that
\begin{equation}\label{3.2}
\begin{aligned}
    k_n^{-1} \left( U^n_h - U^{n-1}_h, X \right) + \left( \nabla U_h^n, \nabla X \right) + \epsilon^{-2} \left( F(U_h^n), X \right) & = \langle f^n,X \rangle, \\
    U_h^0 & = \mathcal{P}_h^0 u^0,
\end{aligned}
\end{equation}
 for every $X \in V^n_{hk}$, 
with  $f^n:= f(t^{}_n)$ 
and $\mathcal{P}^{n}_h$ denoting the orthogonal $L^{}_2$-projection operator onto $V^{n}_{hk}$.


Let now $\Delta^n_h  : V^n_{hk} \rightarrow V^n_{hk}$ defined by $ \left( - \Delta^n_h V, X \right) = \left( \nabla V , \nabla X  \right)$, for all $ V,X \in V^n_{hk}$, i.e., the discrete Laplacian. This allows for the strong representation of \eqref{3.2} as
\begin{equation}\label{3.3}
\begin{aligned}
    k^{-1}_n \left( U_h^{n} - \mathcal{P}^n_{h} U^{n-1}_h \right) - \Delta^n_h U^n_h + \epsilon^{-2} \  \mathcal{P}^n_h F(U^n_h) = \mathcal{P}^n_h f^n.
\end{aligned}
\end{equation}

We now introduce a variant of the elliptic reconstruction \cite{Makridakis-Nochetto2003,Lakkis-Makridakis2006, georgoulis_lakkis_wihler}, which will be instrumental in the proof of the \emph{a posteriori} error bounds below.
\begin{definition}[elliptic reconstruction]\label{def_ellrec} 
	For each $ n= 0,1,\ldots,N,$ we define the elliptic reconstruction $ \omega^n \in H^{1}_0(\Omega)$  to be the solution of the elliptic problem
\begin{equation}\label{4.3}
\begin{aligned}
   \left( \nabla \omega^n, \nabla v \right) = \langle g_h^{n}, v \rangle, \  \ \text{ for all} \ v \in H_0^1(\Omega),
\end{aligned}
\end{equation}
where
\begin{equation}\label{g_h^n}
\begin{aligned}
g^n_h :& = -\Delta_h^{n}U_h^{n} - {\epsilon}^{-2} \left(F(U_h^{n}) -\mathcal{P}_h^n F(U_h^{n}) \right) - \mathcal{P}_h^n f^n + f^n \\
& \quad \quad -  k^{-1}_n \left( \mathcal{P}^n_h  U^{n-1}_h - U^{n-1}_h  \right).
\end{aligned}
\end{equation}
\end{definition}

\begin{remark}[Galerkin orthogonality] \label{remark_GO} We observe that $\omega^n$ satisfies
\begin{align}\label{4.4}
    \left( \nabla ( \omega^n -U^{n}_h \right) , \nabla X ) = 0,  \ \ \text{ for all} \ X \in V^n_{hk}.
\end{align}
This relation implies that $\omega^n - U^{n}_h$ is orthogonal to $V^n_{hk}$ with respect to the Dirichlet inner product, a crucial property that allows to use \emph{a posteriori} error bounds for elliptic problems to estimate various norms of $\omega^n - U^{n}_h$ from above; we refer to Section \ref{sec:fully_computable} for a detailed discussion.
\end{remark}

\begin{definition}[time reconstruction]\label{def_trec} 
	For $ t \in J_n$, $n=1,\ldots,N$, we set
\begin{equation}\label{4.5}
    U^{}_h(t)  := \ell^{}_{n-1}(t) U^{n-1}_h + \ell^{}_{n}(t) U^n_h ,\quad\text{and}\quad
    \omega(t)  := \ell^{}_{n-1}(t) \omega^{n-1} + \ell^{}_{n}(t) \omega^n,
\end{equation}
where $\ell^{}_n$ the piecewise linear Lagrange basis function with $\ell_n^{}(t_k)=\delta_{kn}$. 
\end{definition}
The above definition implies that the time derivative of $U^{}_h$,
\begin{align}\label{4.6}
U^{}_{h,t}(t) = \frac{U^n_h - U^{n-1}_h}{k_n},
\end{align}
is the discrete backward difference at $t^{}_n$.


\section{\emph{A posteriori} error estimates}\label{sec:apost}

We shall now use the reconstructions defined above, together with non-standard energy and continuation arguments and a spectral estimate for the linearized steady-state problem about the approximate solution $U_h$, to arrive at \emph{a posteriori} error bounds in the $L_4(L_4)$-, $L_2(H^1)$- and $L_{\infty}(L_2)$-norms. 

\subsection{Error relation} We begin by splitting the total error as follows:
\begin{equation}\label{error-representation}
    e:=  u  - U_h^{}=\theta -  \rho, \quad \text{where}\quad 
       \theta := \omega - U_h^{}, \ \rho : = \omega -u.
\end{equation}
In view of Remark \ref{remark_GO}, $\theta$ can be estimated by \emph{a posteriori} error bounds for elliptic problems in various norms. 

Also, $\rho$ satisfies an equation of the form \eqref{2.1} with a fully computable right-hand side that consists of  $\theta$  and the problem data.  To see this, \eqref{2.1} along with Definitions \ref{def_ellrec} and \ref{def_trec} and elementary manipulations lead to the following result.
\begin{lemma}[error equation] On $J^{}_n$, $n = 1,\ldots,N$ and for all $v \in H_0^1(\Omega)$, we have
\begin{equation}\label{4.7}
\begin{aligned}
    &\ \langle \rho_t^{},v \rangle   +   \left( \nabla \rho, \nabla v \right)  +  {\epsilon}^{-2}  \left( F(U_h^{})- F(u), v \right) \\
    =&\ \langle f^n - f,v \rangle + \langle \theta^{}_t , v \rangle  + \epsilon^{-2} \left( F(U^{}_h) - F(U^n_h), v \right) + \left( \nabla(\omega -\omega^n) , \nabla v \right) .
\end{aligned}
\end{equation}
\end{lemma}
Therefore, norms of $\rho$ can be estimated through PDE stability arguments; this will be performed below. Before doing so, however, we further estimate the term involving the elliptic reconstructions on the right-hand side from \eqref{4.7}.
\begin{lemma} \label{lem:ellrec_bound} On $J_n$, $n=1,\dots, N$, we have
\begin{equation}\label{L1}
\begin{aligned}
   \left( \nabla(\omega -\omega^n) , \nabla v \right) & \leq  \Big( \norm{\partial{U}_h^n - \partial{U}_h^{n-1}}^{}_{L^{}_2(\Omega)} +  \epsilon^{-2} \norm{F(U^{n}_h) - F(U^{n-1}_h)}^{}_{L^{}_2(\Omega)}\\
   & \quad  \quad  \ + \norm{f^n - f^{n-1}}^{}_{L^{}_2(\Omega)} \Big) \norm{v}^{}_{L^{}_2(\Omega)}
\end{aligned}
\end{equation}
for all $v \in H^1_0(\Omega)$.
\end{lemma}
\begin{proof}
From \eqref{4.5} and Definition \ref{def_ellrec}, we can write
\begin{align*}
    & ( \nabla(\omega -\omega^n) , \nabla v ) 
    = \ell^{}_{n-1}(t) \left( \nabla(\omega^{n-1} -\omega^n) , \nabla v \right) \\ =&\  \ell^{}_{n-1}(t) \left( g_h^{n-1} - g_h^n , v \right) 
   \leq   \norm{g_h^{n-1} - g_h^n}^{}_{L^{}_2(\Omega)} \norm{v}^{}_{L^{}_2(\Omega)}.
\end{align*}
Then, using \eqref{g_h^n} in conjunction with \eqref{3.3}, we obtain
\begin{align*}
   g_h^n 
   & = - k^{-1}_n \left( U_h^{n} - \mathcal{P}^n_{h} U^{n-1}_h \right)  - {\epsilon}^{-2} F(U_h^{n}) + f^n - k^{-1}_n \left( \mathcal{P}^n_h  U^{n-1}_h - U^{n-1}_h  \right) \\
   & = k^{-1}_n \left( U^{n-1}_h - U^{n}_h \right) - {\epsilon}^{-2} F(U_h^{n}) + f^n,
\end{align*}
and correspondingly for $g_h^{n-1}$. Combining the above, the result already follows.
\end{proof}


\subsection{Energy argument}\label{energy-argument}
We begin by introducing some notation. We define
\begin{equation*}
\begin{aligned}
\mathcal{L}^{}_1 : & = \norm{\partial{U}_h^n - \partial{U}_h^{n-1}}^2_{L^{}_2(\Omega)} +  \epsilon^{-4} \norm{F(U^{n}_h) - F(U^{n-1}_h)}^2_{L^{}_2(\Omega)} + \norm{f^n - f^{n-1}}^2_{L^{}_2(\Omega)}, \\
\mathcal{L}^{}_2  :&= \norm{f^n - f}^2_{L^{}_2(\Omega)} +\epsilon^{-4} \norm{F(U^{}_h) - F(U^n_h)}^2_{L^{}_2(\Omega)},
\end{aligned}
\end{equation*}
on each  $J_n$, $n=1,\dots, N$, noting that $\mathcal{L}_2^{}\equiv\mathcal{L}_{2}^{}(t)$; for $n=1$ we adopt the convention that $ U^{-1}_{h} = U^0_{h}$.

Moreover, for brevity, we also set
\begin{equation*}
\begin{aligned}
\Theta^{}_1\equiv\Theta^{}_1(t): & = \frac{1}{2}  \norm{\theta^{}_t}^2_{L^{}_2(\Omega)} + \frac{11 }{4}C^4_{PF} \norm{\theta^{}_t}^4_{L^{}_4(\Omega)}   ,\\
\Theta^{}_2\equiv\Theta^{}_2(t): & = \epsilon^{-4}\Big( \big(C_0+ 396 \norm{U^{}_h}^2_{L^{}_\infty(\Omega)}\big) \norm{\theta}^2_{L^{}_2(\Omega)} 
  + \frac{ C^{}_1}{2 }  \norm{\theta}^4_{L^{}_4(\Omega)} +  C^{}_{0} \norm{\theta}^6_{L^{}_6(\Omega)} \Big),
\end{aligned}
\end{equation*}
\begin{equation*}
\begin{aligned}
A(t): & =   \epsilon^{-2}\Big( (  \theta^2 \rho^2+\rho^4+|\nabla \rho|^2, \int_{t}^{\uptau} \rho^2(s) \diff s   )  +  (  \theta^2, \rho^2 )\Big),
\end{aligned}
\end{equation*}
where $ C^{}_{0} : = (C^{}_{PF} \tilde{c}^2+1)/2$, $ C^{}_1 : =9 + 9  C^{}_{PF} \tilde{c}^2 +   6^4 11^2 C^{2}_{PF} \tilde{c}^4$, $C^{}_2 :=  2\cdot3^7 C^{2}_{PF} \tilde{c}^4$, where $C^{}_{PF}$ is the constant of the Poincar\'{e}-Friedrichs inequality $\norm{v}\le C_{PF}\norm{\nabla v}$ and $\tilde{c}$ as in \eqref{eq:GNL}. 

\begin{lemma}[$d=2$]\label{lemma_4.4}
Let $d=2$ and $u$ be the  solution of \eqref{2.1} and $\omega$ as in \eqref{4.5}. Assume that $\rho(t) \in W^{1,4}_0(\Omega)$ for a.e. $ t \in (0,T]$. Then, for any $\uptau\in (0,T]$, we have
\begin{equation}\label{4.10}
\begin{aligned}
    &\quad \frac{1}{4} \int_{0}^{\uptau} \norm{\rho}^4_{L^{}_4(\Omega)} \diff t  +  \frac{1}{8} \norm{\int_{0}^{\uptau}  \nabla \rho^2(s) \diff s}^2_{L^{}_2(\Omega)}  +  \frac{1}{2}  \norm{\rho(\uptau)}^2_{L^{}_2(\Omega)} \\
     & +  \int_{0}^{\uptau} A(t) \diff t +  \int_{0}^{\uptau} \Big(  \big( 1- \frac{\epsilon^2}{2}\big)\norm{\nabla \rho}^2_{L^{}_2(\Omega)} + \frac{1}{\epsilon^{2}} \left( F'(U^{}_h) \rho, \rho  \right) \Big) \diff t \\
     \leq & \quad \frac{1}{2}  \norm{\rho(0)}^2_{L^{}_2(\Omega)}  + \frac{C_{PF}^2}{2} \norm{\rho(0)}^4_{L_4^{}(\Omega)} + \int_{0}^{\uptau} \big( \Theta^{}_1 +  \Theta^{}_2 +  C_0( \mathcal{L}^{}_1 + \mathcal{L}^{}_2 )\big) \diff t  \\
    &   + \frac{1}{2}\int_{0}^{\uptau}\Big(  \norm{\int_{t}^{\uptau} \nabla \rho^2(s) \diff s}^2_{L_2^{}(\Omega)} +\alpha(U_h) \norm{\rho}^2_{L^{}_2(\Omega)}\Big)  \diff t \\
    & +\frac{1}{4\epsilon^{6}} \int_{0}^{\uptau}\Big( \beta(\theta,U_h) \norm{\int_{t}^{\uptau} \nabla \rho^2(s) \diff s}^4_{L_2^{}(\Omega)} +\gamma(\theta,U_h)\norm{\rho}^4_{L^{}_2(\Omega)}\Big) \diff t,
\end{aligned}
\end{equation}
where
\[
\begin{aligned}
\alpha(U_h):=& \norm{F'(U_h^{})}^2_{L_{\infty}(\Omega)}  + \norm{U_h^{}}^2_{L_{\infty}(\Omega)} + 7 \\
\beta(\theta,U_h):= &\ \frac{C^{}_2\epsilon^{4}}{16}  \big(  \norm{\theta}^4_{L_{\infty}(\Omega)} + \norm{U^{}_h}^4_{L_{\infty}(\Omega)} \big)  + 2\epsilon^{2} \norm{U^{}_h}^4_{L_{\infty}(\Omega)}  \\
&
+2 C_{PF}^2 \tilde{c}^4 \norm{F'(U_h^{})}^2_{L_{\infty}(\Omega)}+ 11 \epsilon^{6} \big(\norm{F'(U_h^{})}^4_{L_{\infty}(\Omega)} +  \norm{U^{}_h}^4_{L_{\infty}(\Omega)} +6\big) ,\\
\gamma(\theta,U_h):= &\ 2 \tilde{c}^4\Big( C_{PF}^2  \norm{F'(U^{}_h)}^2_{L^{}_{\infty}(\Omega)} +  36 \big(  \norm{\theta}^2_{L^{}_{\infty}(\Omega)} +\norm{U^{}_h}^2_{L^{}_{\infty}(\Omega)} \big)\Big) .
\end{aligned}
\]
\end{lemma}
\begin{proof}
	Using Taylor's theorem, we immediately deduce
	\[
	F(U_h^{}) - F(u) = -e F'(U_h^{}) -3e^2 U_h^{} -e^3.
	\] 
Let $\phi:[0,\uptau]\times \Omega \rightarrow \mathbb{R}$, φορ $0<\uptau \leq T$, such that
\begin{equation}\label{4.11}
    \phi(\cdot,t)  = \rho(\cdot,t) \Big( \int_{t}^{\uptau} \rho^2(\cdot,s) \diff s +  1 \Big), \quad t\in [0,\uptau].
\end{equation}
Hypothesis $\rho \in W^{1,4}_0(\Omega)$ implies that $\phi \in H_0^1(\Omega)$. Setting $v=\phi$ in \eqref{4.7}, we have
\begin{align*}
    & \langle \rho_t^{},\phi \rangle   +   \left( \nabla \rho, \nabla \phi \right)  -  {\epsilon}^{-2}  \left(e F'(U_h^{}) +3e^2 U_h^{} +e^3, \phi \right) = \langle f^n - f,\phi \rangle + \langle \theta^{}_t , \phi \rangle \\[0.1cm]
    & \quad + \epsilon^{-2} \left( F(U^{}_h) - F(U^n_h), \phi \right)  + \left( \nabla(\omega -\omega^n) , \nabla \phi \right)  .
\end{align*}
Observing now the identities
\begin{align*}
\left( e^2 U^{}_h ,  \phi \right)  = & \left( \theta^2 U^{}_h ,  \phi \right)  +  \left( \rho^2 U^{}_h ,   \phi \right)   -  2 \left( \theta \rho  U^{}_h ,  \phi \right), \\
\left( e^3,  \phi \right)  = & \left( \theta^3, \phi \right)  - 3 \left(  \theta^2 \rho,  \phi \right) +  3 \left(  \theta \rho^2, \phi \right)  -  \left( \rho^3, \phi \right),
\end{align*}
elementary calculations yield
\begin{equation}\label{4.12}
\begin{aligned}
    & \quad\ \frac{1}{2}\frac{\diff}{\diff{ t}} \norm{\rho}^2_{L^{}_2(\Omega)} + \langle \rho_t^{}, \rho \int_{t}^{\uptau} \rho^2(s) \diff s \rangle +  ( \nabla \rho, \rho   \int_{t}^{\uptau} \nabla \rho^2(s) \diff s)   \\
    &  \quad  + \norm{\nabla \rho}^2_{L^{}_2(\Omega)} + \epsilon^{-2} \left( F'(U^{}_h) \rho, \rho  \right) + \epsilon^{-2}\norm{\rho}^4_{L^{}_4(\Omega)} + A(t) \\
    & =  \langle f^n - f,\phi \rangle +  \langle \theta^{}_t, \phi \rangle   + \epsilon^{-2} \left( F(U^{}_h) - F(U^n_h), \phi \right) + \left( \nabla(\omega -\omega^n) , \nabla \phi \right) \\
    & \quad + 3 \epsilon^{-2} \left(  \theta^2 U^{}_h , \phi \right)   + 3 \epsilon^{-2} \left( \rho^2 U^{}_h , \phi \right) - 6 \epsilon^{-2} \left( \theta \rho U^{}_h, \phi \right)  +  \epsilon^{-2} ( \theta^3, \phi )  \\
    & \quad  + 3 \epsilon^{-2} ( \theta \rho^2,\phi )  \!+\!  \epsilon^{-2} ( F'(U^{}_h) \theta, \phi  ) - \epsilon^{-2} ( F'(U^{}_h) \rho, \rho \!\!\int_{t}^{\uptau} \rho^2(s) \diff s  )  =: \sum_{j=1}^{11} I^{}_{j}.
\end{aligned}
\end{equation}
We shall further estimate each $I_j$. We begin by splitting $I^{}_1 $ into
\begin{align*}
I^{}_1  = \langle f^n - f, \rho \int_{t}^{\uptau} \rho^2(s) \diff s \rangle + \langle f^n - f, \rho \rangle =: I^1_1 + I^2_1.
\end{align*}
Applying  H\"{o}lder, GNL for $d=2$, Poincar\'{e}-Friedrichs  and Young inequalities gives, respectively,
\begin{align*}
    I^1_1 &  \leq   \norm{ f^n - f }^{}_{L^{}_2(\Omega)} \norm{\rho}^{}_{L^{}_4(\Omega)} \norm{\int_{t}^{\uptau} \rho^2(s) \diff s}^{}_{L_4^{}(\Omega)} \\
    & \leq  \tilde{c} \norm{ f^n - f }^{}_{L^{}_2(\Omega)} \norm{\rho}^{}_{L^{}_4(\Omega)}  \norm{\int_{t}^{\uptau} \rho^2(s) \diff s}^{1/2}_{L_2^{}(\Omega)} \norm{ \nabla \int_{t}^{\uptau} \rho^2(s) \diff s}^{1/2}_{L_2^{}(\Omega)} \\
    & \leq  C^{1/2}_{PF} \tilde{c} \ \norm{ f^n - f }^{}_{L^{}_2(\Omega)} \norm{\rho}^{}_{L^{}_4(\Omega)} \norm{ \nabla \int_{t}^{\uptau} \rho^2(s) \diff s}^{}_{L_2^{}(\Omega)} \\
    & \leq \frac{C^{}_{PF} \tilde{c}^2}{2}  \norm{ f^n - f }^{}_{L^{}_2(\Omega)} + \frac{1}{44} \norm{\rho}^4_{L^{}_4(\Omega)} + \frac{11}{4} \norm{\int_{t}^{\uptau} \nabla \rho^2(s) \diff s}^4_{L_2^{}(\Omega)}.
\end{align*}
The Cauchy-Schwarz and Young inequalities also yield
$
 I^2_1 \leq 
  \frac{1}{2} \mathcal{L}^{}_2 + \frac{1}{2} \norm{\rho}^2_{L^{}_2(\Omega)}.
$ 
Likewise, we split $I_3$ as follows:
$$I^{}_3 = \epsilon^{-2} (  F(U^{}_h) - F(U^n_h), \rho \int_{t}^{\uptau} \rho^2(s) \diff s ) +  \epsilon^{-2} \left(  F(U^{}_h) - F(U^n_h), \rho \right) =: I^1_3 + I^2_3,
$$
yielding the following bounds
\begin{align*}
    I^1_3 \leq & \ \frac{C^{}_{PF} \tilde{c}^2}{2 \epsilon^{4}} \norm{F(U^{}_h) - F(U^n_h)}^2_{L^{}_2(\Omega)}{+} \frac{1}{44} \norm{\rho}^4_{L^{}_4(\Omega)} {+} \frac{11}{4} \norm{\int_{t}^{\uptau} \nabla \rho^2(s) \diff s}^4_{L_2^{}(\Omega)},\\
    I^2_3 \leq &\ \frac{1}{2 \epsilon^{4}} \norm{F(U^{}_h) - F(U^n_h)}^2_{L^{}_2(\Omega)} + \frac{1}{2} \norm{\rho}^2_{L^{}_2(\Omega)}.
\end{align*}
From Lemma \ref{lem:ellrec_bound} and working as before, we have
$$I^{}_4 = ( \nabla(\omega -\omega^n) , \nabla \Big( \rho \int_{t}^{\uptau} \rho^2(s) \diff s \Big)  ) + \left( \nabla(\omega -\omega^n) , \nabla  \rho \right) := I^1_4 + I^2_4,$$
\begin{align*}
    I^1_{4} \leq & \ \frac{ C^{}_{PF} \tilde{c}^2 }{2} \mathcal{L}^{}_1  + \frac{3}{44} \norm{\rho}^4_{L^{}_4(\Omega)} + \frac{11}{4} \norm{\int_{t}^{\uptau} \nabla \rho^2(s) \diff s}^4_{L_2^{}(\Omega)},\\
    I^2_4 \leq & \ \frac{1}{2} \mathcal{L}^{}_1 + \frac{3}{2} \norm{\rho}^2_{L^{}_2(\Omega)}.
\end{align*}
Next, we split $I_2$ as follows:
$$I^{}_2 = \langle \theta^{}_t, \rho \int_{t}^{\uptau} \rho^2(s) \diff s \rangle + \langle \theta^{}_t, \rho \rangle =: I^1_2 + I^2_2$$
and, using  H\"{o}lder, Poincar\'{e}-Friedrichs and Young inequalities, we deduce
\begin{align*}
   I^1_2 & \leq    \norm{\theta^{}_t}^{}_{L^{}_4(\Omega)} \norm{\rho}^{}_{L^{}_4(\Omega)} \norm{\int_{t}^{\uptau} \rho^2(s) \diff s}^{}_{L_2^{}(\Omega)}\\
   & \leq   C^{}_{PF} \norm{\theta^{}_t}^{}_{L^{}_4(\Omega)} \norm{\rho}^{}_{L^{}_4(\Omega)} \norm{\nabla \int_{t}^{\uptau}  \rho^2(s) \diff s}^{}_{L_2^{}(\Omega)}\\
    & \leq \frac{11 C^4_{PF}}{4}  \norm{\theta^{}_t}^4_{L^{}_4(\Omega)} + \frac{1}{44}  \norm{\rho}^4_{L^{}_4(\Omega)} + \frac{1}{2} \norm{\int_{t}^{\uptau} \nabla \rho^2(s) \diff s}^2_{L_2^{}(\Omega)},\\
    I^2_2 & \leq    \frac{1}{2} \norm{\theta^{}_t}^2_{L^{}_2(\Omega)} + \frac{1}{2} \norm{\rho}^2_{L^{}_2(\Omega)}.
\end{align*}
Next, we split
$$I^{}_5 = 3 \epsilon^{-2} (  \theta^2 U^{}_h , \rho \int_{t}^{\uptau} \rho^2(s) \diff s ) +  3  \epsilon^{-2} (  \theta^2 U^{}_h , \rho  )=: I^1_5 + I^2_5 ,$$
which can be further bounded as follows: 
\begin{align*}
   I^1_5 & \leq 3 \epsilon^{-2} \norm{\theta^2}^{}_{L^{}_2(\Omega)} \norm{U_h^{}}^{}_{L_{\infty}(\Omega)} \norm{\rho}^{}_{L^{}_4(\Omega)} \norm{\int_{t}^{\uptau}  \rho^2(s) \diff s}^{}_{L_4^{}(\Omega)}\\
   & \leq  3 \epsilon^{-2} \tilde{c} \ \norm{\theta}^2_{L^{}_4(\Omega)} \norm{U_h^{}}^{}_{L_{\infty}(\Omega)} \norm{\rho}^{}_{L^{}_4(\Omega)} \norm{\int_{t}^{\uptau}  \rho^2(s) \diff s}^{1/2}_{L_2^{}(\Omega)} \norm{ \nabla \int_{t}^{\uptau}  \rho^2(s) \diff s}^{1/2}_{L_2^{}(\Omega)}\\
   & \leq  3 \epsilon^{-2} C^{1/2}_{PF} \tilde{c} \ \norm{\theta}^2_{L^{}_4(\Omega)} \norm{U_h^{}}^{}_{L_{\infty}(\Omega)} \norm{\rho}^{}_{L^{}_4(\Omega)}  \norm{ \nabla \int_{t}^{\uptau}  \rho^2(s) \diff s}^{}_{L_2^{}(\Omega)}\\
   & \leq \frac{9 C^{}_{PF} \tilde{c}^2 }{2\epsilon^{4}}  \norm{\theta}^4_{L^{}_4(\Omega)}   + \frac{1}{44} \norm{\rho}^4_{L^{}_4(\Omega)} +  \frac{11}{4}  \norm{U_h^{}}^4_{L_{\infty}(\Omega)}  \norm{\int_{t}^{\uptau} \nabla \rho^2(s) \diff s}^4_{L_2^{}(\Omega)}, \\
   I^2_5 & \leq   \frac{9 }{2\epsilon^{4}}  \norm{\theta}^4_{L^{}_4(\Omega)} +  \frac{1}{2}  \norm{U_h^{}}^2_{L_{\infty}(\Omega)}  \norm{\rho}^2_{L^{}_2(\Omega)}.
\end{align*}
In the same spirit, we also have
$$ I^{}_7 = - 6 \epsilon^{-2} ( \theta \rho U^{}_h, \rho  \int_{t}^{\uptau} \rho^2(s) \diff s )  - 6 \epsilon^{-2} ( \theta \rho U^{}_h, \rho )=: I^1_7 + I^2_7,$$
and, thus,
\begin{align*}
    I^1_7 & \leq    6 \epsilon^{-2} \norm{\theta}^{}_{L^{}_4(\Omega)} \norm{\rho^2}^{}_{L^{}_2(\Omega)} \norm{U_h^{}}^{}_{L_{\infty}(\Omega)} \norm{\int_{t}^{\uptau}  \rho^2(s) \diff s}^{}_{L_4^{}(\Omega)} \\
    & \leq  \frac{6 C^{1/2}_{PF}\tilde{c}}{ \epsilon^{2}}  \ \norm{\theta}^{}_{L^{}_4(\Omega)} \norm{\rho}^2_{L^{}_4(\Omega)} \norm{U_h^{}}^{}_{L_{\infty}(\Omega)} \norm{\int_{t}^{\uptau} \nabla \rho^2(s) \diff s}^{}_{L_2^{}(\Omega)}  \\
    & \leq   \frac{ 6^4 11^2 C^{2}_{PF} \tilde{c}^4}{ 2\epsilon^{4}} \norm{\theta}^4_{L^{}_4(\Omega)} + \frac{1}{44} \norm{\rho}^4_{L^{}_4(\Omega)}  +  \frac{1}{2 \epsilon^{4}} \norm{U_h^{}}^4_{L_{\infty}(\Omega)}  \norm{\int_{t}^{\uptau} \nabla \rho^2(s) \diff s}^4_{L_2^{}(\Omega)} ,\\
    I^2_7 & \leq  \frac{6}{ \epsilon^{2}} \norm{\theta}^{}_{L^{}_2(\Omega)} \norm{U_h^{}}^{}_{L_{\infty}(\Omega)} \norm{\rho}^2_{L^{}_4(\Omega)}
    \leq \frac{396}{\epsilon^{4}} \norm{U_h^{}}^2_{L_{\infty}(\Omega)} \norm{\theta}^2_{L^{}_2(\Omega)} + \frac{1}{44} \norm{\rho}^4_{L^{}_4(\Omega)}.
\end{align*}
Next, we consider the splitting
$$ I^{}_{10} = \epsilon^{-2} ( F'(U^{}_h) \theta, \rho \int_{t}^{\uptau} \rho^2(s) \diff s  ) +\epsilon^{-2} ( F'(U^{}_h) \theta, \rho  ) =: I^1_{10} + I^2_{10},$$ 
and we have the following bounds:
\begin{align*}
    I^1_{10} & \leq    \epsilon^{-2}  \norm{F'(U_h^{})}^{}_{L_{\infty}(\Omega)} \norm{\theta}^{}_{L^{}_2(\Omega)} \norm{\rho}^{}_{L^{}_4(\Omega)} \norm{\int_{t}^{\uptau}  \rho^2(s) \diff s}^{}_{L_4^{}(\Omega)}\\
    & \leq  \frac{C^{1/2}_{PF} \tilde{c}}{\epsilon^{2} } \norm{F'(U_h^{})}^{}_{L_{\infty}(\Omega)} \norm{\theta}^{}_{L^{}_2(\Omega)} \norm{\rho}^{}_{L^{}_4(\Omega)} \norm{\int_{t}^{\uptau} \nabla \rho^2(s) \diff s}^{}_{L_2^{}(\Omega)} \\
    & \leq  \frac{C^{}_{PF} \tilde{c}^2}{2 \epsilon^{4}} \norm{\theta}^2_{L^{}_2(\Omega)}  +  \frac{1}{44}  \norm{\rho}^4_{L^{}_4(\Omega)}  + \frac{11}{4} \norm{F'(U_h^{})}^{4}_{L_{\infty}(\Omega)} \norm{\int_{t}^{\uptau} \nabla \rho^2(s) \diff s}^4_{L_2^{}(\Omega)},\\
    I^2_{10} & \leq    \frac{1}{2 \epsilon^{4}} \norm{\theta}^2_{L^{}_2(\Omega)} + \frac{1}{2}  \norm{F'(U_h^{})}^2_{L_{\infty}(\Omega)}  \norm{\rho}^2_{L^{}_2(\Omega)}.
\end{align*}
Next, we set
$$ I^{}_8 = \epsilon^{-2} ( \theta^3,\rho \int_{t}^{\uptau} \rho^2(s) \diff s )  +  \epsilon^{-2} ( \theta^3,\rho ) =: I^1_8 + I^2_8,$$
and we further estimate as follows:
\begin{align*}
    I^1_8 & \leq   \epsilon^{-2}  \norm{\theta^3}^{}_{L^{}_2(\Omega)} \norm{\rho}^{}_{L^{}_4(\Omega)} \norm{\int_{t}^{\uptau} \rho^2(s) \diff s}^{}_{L_4^{}(\Omega)}\\
    & \leq  \frac{ C^{1/2}_{PF} \tilde{c}}{\epsilon^{2}}  \norm{\theta}^{3}_{L^{}_6(\Omega)} \norm{\rho}^{}_{L^{}_4(\Omega)} \norm{\int_{t}^{\uptau} \nabla \rho^2(s) \diff s}^{}_{L_2^{}(\Omega)} \\
    & \leq  \frac{ C^{}_{PF} \tilde{c}^2 }{2 \epsilon^{4}} \norm{\theta}^6_{L^{}_6(\Omega)}  +  \frac{1}{44} \norm{\rho}^4_{L^{}_4(\Omega)}+  \frac{11}{4}  \norm{\int_{t}^{\uptau} \nabla \rho^2(s) \diff s}^4_{L_2^{}(\Omega)},\\
    I^2_8 & \leq   \frac{1}{2 \epsilon^{4}} \norm{\theta}^6_{L^{}_6(\Omega)}    +  \frac{1}{2}  \norm{\rho}^2_{L^{}_2(\Omega)}.
\end{align*}
For $I^{}_6$ and $I^{}_9$, we work collectively as follows:
\begin{align*}
    I^{}_6+I^{}_9 & = 3 \epsilon^{-2} ( \rho^2 (U^{}_h+\theta) ,  \rho \int_{t}^{\uptau} \rho^2(s) \diff s) + 3 \epsilon^{-2} ( \rho^2 (U^{}_h+\theta),  \rho  ) =: I^1_{6,9} + I^2_{6,9},
\end{align*}
and estimate:
\begin{align*}
    I^1_{6,9} 
    & \leq  \frac{3   C^{1/2}_{PF}\tilde{c}}{ \epsilon^{2}} \norm{\rho}^{3}_{L^{}_{4}(\Omega)} \left( \norm{\theta}^{}_{L_{\infty}(\Omega)} + \norm{U^{}_h}^{}_{L_{\infty}(\Omega)} \right) \norm{ \nabla \int_{t}^{\uptau} \rho^2(s) \diff s}^{}_{L_2^{}(\Omega)} \\
    & \leq  \epsilon^{-2} \norm{\rho}^4_{L^{}_4(\Omega)} +  \frac{ C^{}_2
    }{ 64  \epsilon^{2}} \left(  \norm{\theta}^4_{L_{\infty}(\Omega)} + \norm{U^{}_h}^4_{L_{\infty}(\Omega)} \right) \norm{\int_{t}^{\uptau} \nabla \rho^2(s) \diff s}^4_{L_2^{}(\Omega)},\\
   I^2_{6,9}  & \leq  \frac{3}{  \epsilon^{2}} \norm{ \rho}^2_{L^{}_4(\Omega)} \left( \norm{\theta}^{}_{L^{}_{\infty}(\Omega)} +\norm{U^{}_h}^{}_{L^{}_{\infty}(\Omega)} \right)   \norm{\rho}^{}_{L^{}_2(\Omega)} \\
   & \leq  \frac{3  \tilde{c}^2}{\epsilon^{2} }  \norm{ \nabla \rho}^{}_{L^{}_2(\Omega)} \left( \norm{\theta}^{}_{L^{}_{\infty}(\Omega)} +\norm{U^{}_h}^{}_{L^{}_{\infty}(\Omega)} \right)   \norm{\rho}^{2}_{L^{}_2(\Omega)} \\
   & \leq  \frac{\epsilon^2}{4}  \norm{\nabla \rho}^2_{L^{}_2(\Omega)}  +  \frac{18 \tilde{c}^4} {\epsilon^{6}} \left( \norm{\theta}^2_{L^{}_{\infty}(\Omega)} +\norm{U^{}_h}^2_{L^{}_{\infty}(\Omega)} \right)   \norm{\rho}^4_{L^{}_2(\Omega)}.
\end{align*}
Finally for the last term on the right-hand side of \eqref{4.12}, we have
\begin{align*}
    I^{}_{11} & \leq  \epsilon^{-2} \norm{F'(U_h^{})}^{}_{L_{\infty}(\Omega)} \norm{\rho}^2_{L^{}_4(\Omega)} \norm{\int_{t}^{\uptau} \rho^2(s) \diff s}^{}_{L_2^{}(\Omega)} \\
    & \leq  \frac{ C^{}_{PF} \tilde{c}^2}{\epsilon^{2}} \norm{F'(U_h^{})}^{}_{L_{\infty}(\Omega)} \norm{\rho}^{}_{L^{}_2(\Omega)} \norm{ \nabla \rho}^{}_{L^{}_2(\Omega)} \norm{\int_{t}^{\uptau} \nabla \rho^2(s) \diff s}^{}_{L_2^{}(\Omega)}\\
    & \leq   \frac{\epsilon^2}{4}  \norm{\nabla \rho}^2_{L^{}_2(\Omega)}  +  \frac{  C^{2}_{PF} \tilde{c}^4 }{ 2 \epsilon^{6}} \norm{F'(U^{}_h)}^2_{L^{}_{\infty}(\Omega)}  \norm{ \rho}^4_{L^{}_2(\Omega)} \\
    & \quad   +  \frac{ C^{2}_{PF} \tilde{c}^4 }{2 \epsilon^{6}} \norm{F'(U^{}_h)}^2_{L^{}_{\infty}(\Omega)} \norm{\int_{t}^{\uptau} \nabla \rho^2(s) \diff s}^4_{L^{}_2(\Omega)}. 
\end{align*}
Applying the above estimates into \eqref{4.12} and integrating with respect to $t\in(0,\uptau)$ and observing the identities
\begin{align*}
    \int_0^{\tau} \langle \rho_t^{}, \rho \int_{t}^{\uptau} \rho^2(s) \diff s \rangle \diff t & = - \frac{1}{2} \langle \rho^2(0), \rho \int_{0}^{\uptau} \rho^2(s) \diff s \rangle + \frac{1}{2} \int_{t}^{\uptau} \norm{ \rho}^4_{L^{}_4(\Omega)} \diff t,\\
    \int_0^{\tau} ( \nabla \rho, \rho   \int_{t}^{\uptau} \nabla \rho^2(s) \diff s) \diff t & = 
    -\frac{1}{4} \int_0^{\tau} \frac{\diff }{\diff t} (  \int_{t}^{\uptau} \nabla \rho^2(s) \diff s, \int_{t}^{\uptau} \nabla \rho^2(s) \diff s ) \diff t \\
    & =\frac{1}{4}\norm{\int_{0}^{\uptau} \nabla \rho^2(s) \diff s}^2_{L^{}_2(\Omega)},
\end{align*}
along with elementary manipulations, the result already follows.
\end{proof}

The use of the dimension-dependent GNL inequalities \eqref{eq:GNL} necessitates certain modifications in the above argument when $d=3$, which we now provide. For brevity, we shall only provide the terms which are handled differently to the proof of the two-dimensional case from Lemma \ref{lemma_4.4}.

\begin{lemma}[$d=3$]\label{lemma_4.5}
Let $d=3$, $u$ the solution of \eqref{2.1} and $\omega$ as in \eqref{4.5}. Assume that $\rho(t) \in W^{1,4}_0(\Omega)$ for a.e. $ t \in (0,T]$. Then, for any $\uptau\in (0,T]$, we have
\begin{equation}\label{d=3}
\begin{aligned}
  &\quad \frac{1}{8} \int_{0}^{\uptau} \norm{\rho}^4_{L^{}_4(\Omega)} \diff t  +  \frac{1}{8} \norm{\int_{0}^{\uptau}  \nabla \rho^2(s) \diff s}^2_{L^{}_2(\Omega)}  +  \frac{1}{2}  \norm{\rho(\uptau)}^2_{L^{}_2(\Omega)} \\
 &+  \int_{0}^{\uptau} A(t) \diff t +  \int_{0}^{\uptau} \Big(  \big( 1- \frac{\epsilon^2}{2}\big)\norm{\nabla \rho}^2_{L^{}_2(\Omega)} + \frac{1}{\epsilon^{2}} ( F'(U^{}_h) \rho, \rho  ) \Big) \diff t \\
    \leq&  \quad \frac{1}{2}  \norm{\rho(0)}^2_{L^{}_2(\Omega)}  + \frac{C_{PF}^2}{2} \norm{\rho(0)}^4_{L_4^{}(\Omega)} + \int_{0}^{\uptau} \big( \Theta^{}_1 + \tilde{\Theta}^{}_2 +\tilde{C}_0( \mathcal{L}^{}_1 + \mathcal{L}^{}_2 ) \big) \diff t  \\
    &   + \frac{1}{2}\int_{0}^{\uptau}  \Big( \norm{\int_{t}^{\uptau} \nabla \rho^2(s) \diff s}^2_{L_2^{}(\Omega)} +(\alpha(U_h)+1)\norm{\rho}^2_{L^{}_2(\Omega)}\Big)\diff t \\
    &  +\frac{1}{4\epsilon^{10}} \int_{0}^{\uptau} \Big(\tilde{\beta}(\theta,U_h) \norm{\int_{t}^{\uptau} \nabla \rho^2(s) \diff s}^4_{L_2^{}(\Omega)} +\tilde{\gamma}(\theta,U_h) \norm{\rho}^4_{L^{}_2(\Omega)} \Big)\diff t ,
\end{aligned}
\end{equation}
where
\[
\begin{aligned}
\tilde{\Theta}^{}_2:  =&\ \epsilon^{-4}\Big( \big(\tilde{C}_0+ 396 \norm{U^{}_h}^2_{L^{}_\infty(\Omega)}\big) \norm{\theta}^2_{L^{}_2(\Omega)} 
+ \frac{ \tilde{C}^{}_1}{2 }  \norm{\theta}^4_{L^{}_4(\Omega)} +  \tilde{C}^{}_{0} \norm{\theta}^6_{L^{}_6(\Omega)} \Big),\\
\tilde{\beta}(\theta,U_h):= &\ \frac{\tilde{C}^{}_2 \epsilon^{8}}{16}\big(  \norm{\theta}^4_{L_{\infty}(\Omega)} + \norm{U^{}_h}^4_{L_{\infty}(\Omega)} \big)  + 2\epsilon^{6} \norm{U^{}_h}^4_{L_{\infty}(\Omega)}
\\
&   +2 C^{}_{PF} \tilde{c}^4\epsilon^2 \norm{F'(U_h^{})}^4_{L_{\infty}(\Omega)} + 11 \epsilon^{10} \big(\norm{F'(U_h^{})}^4_{L_{\infty}(\Omega)} +  \norm{U^{}_h}^4_{L_{\infty}(\Omega)} +6\big), \\
\tilde{\gamma}(\theta,U_h):=&\ 324  C^{}_{PF}  \tilde{c}^4 \big(  \norm{\theta}^4_{L^{}_{\infty}(\Omega)} +\norm{U^{}_h}^4_{L^{}_{\infty}(\Omega)} \big),
\end{aligned}
\]
with $\tilde{C}_0:= ( C^{1/2}_{PF}  \tilde{c}^2+1)/2$, 
$\tilde{C}_1: = 9 + 9 C^{1/2}_{PF}  \tilde{c}^2 + 6^4 11^2 C^{}_{PF}  \tilde{c}^4$, $\tilde{C}_2:= 3^7C^{}_{PF}  \tilde{c}^4$.
\end{lemma}
\begin{proof}
Starting from \eqref{4.12},  we discuss only the different treatment of the terms $I^{}_j$, $j=6,9,11$; the estimation of the remaining terms is identical to the proof of Lemma \ref{lemma_4.4} and is, therefore, omitted. 
To that end, we begin by setting $\zeta(\theta,U_h):= \norm{\theta}^{}_{L_{\infty}(\Omega)} + \norm{U^{}_h}^{}_{L_{\infty}(\Omega)} $. Then,  we have
\begin{align*}
    I^1_{6,9}   \leq&\ \frac{3}{ \epsilon^{2}} \norm{\rho^3}^{}_{L^{}_{4/3}(\Omega)}\zeta(\theta,U_h)\norm{\int_{t}^{\uptau} \rho^2(s) \diff s}^{}_{L_4^{}(\Omega)} \\
    \leq & \ \frac{3\tilde{c} }{ \epsilon^{2}} \norm{\rho}^{3}_{L^{}_{4}(\Omega)} \zeta(\theta,U_h) \norm{  \int_{t}^{\uptau} \rho^2(s) \diff s}^{1/4}_{L_2^{}(\Omega)} \norm{ \nabla \int_{t}^{\uptau} \rho^2(s) \diff s}^{3/4}_{L_2^{}(\Omega)} \\
     \leq &\ \frac{3\tilde{c} C^{1/4}_{PF}}{ \epsilon^{2}}  \norm{\rho}^{3}_{L^{}_{4}(\Omega)}\zeta(\theta,U_h)\norm{  \int_{t}^{\uptau}\nabla \rho^2(s) \diff s}^{}_{L_2^{}(\Omega)} \\
     \leq& \ \frac{1}{2 \epsilon^{2}} \norm{\rho}^4_{L^{}_4(\Omega)} +  \frac{ \tilde{C}^{}_2
    }{64 \epsilon^{2}} \left(  \norm{\theta}^4_{L_{\infty}(\Omega)} + \norm{U^{}_h}^4_{L_{\infty}(\Omega)} \right) \norm{\int_{t}^{\uptau} \nabla \rho^2(s) \diff s}^4_{L_2^{}(\Omega)},
\end{align*}
using \eqref{eq:GNL} for $d=3$. Similarly, we have
\begin{align*}
     I^2_{6,9}   \leq&\ \frac{3}{ \epsilon^{2}} \norm{\rho}^2_{L^{}_{4}(\Omega)}\zeta(\theta,U_h) \norm{\rho}^{}_{L^{}_{2}(\Omega)} 
     \leq   \frac{3\tilde{c}}{ \epsilon^{2}}  \norm{\rho}^{1/4}_{L^{}_{2}(\Omega)} \norm{\nabla \rho}^{3/4}_{L^{}_{2}(\Omega)} \norm{\rho}^{}_{L^{}_{4}(\Omega)} \zeta(\theta,U_h)\norm{\rho}^{}_{L^{}_{2}(\Omega)} \\
     \leq &\ \frac{3C^{1/4}_{PF} \tilde{c}}{ \epsilon^{2}}     \norm{\nabla \rho}^{}_{L^{}_{2}(\Omega)} \norm{\rho}^{}_{L^{}_{4}(\Omega)} \zeta(\theta,U_h)\norm{\rho}^{}_{L^{}_{2}(\Omega)}\\
     \leq &\ \frac{\epsilon^2}{2} \norm{\nabla \rho}^{2}_{L^{}_{2}(\Omega)} + \frac{18C^{1/2}_{PF} \tilde{c}^2}{ \epsilon^{6}} \norm{\rho}^{2}_{L^{}_{4}(\Omega)} \left(  \norm{\theta}^{2}_{L_{\infty}(\Omega)} + \norm{U^{}_h}^{2}_{L_{\infty}(\Omega)}\right)\norm{\rho}^{2}_{L^{}_{2}(\Omega)} \\
     \leq&\  \frac{\epsilon^2}{2} \norm{\nabla \rho}^{2}_{L^{}_{2}(\Omega)} + \frac{1}{2\epsilon^{2}} \norm{\rho}^4_{L^{}_4(\Omega)} + \frac{81  C^{}_{PF} \tilde{c}^4}{ \epsilon^{10}} \left( \norm{\theta}^{4}_{L_{\infty}(\Omega)} + \norm{U^{}_h}^{4}_{L_{\infty}(\Omega)} \right) \norm{\rho}^{4}_{L^{}_{2}(\Omega)}.
\end{align*}
Likewise, using completely analogous arguments, we have
\begin{align*}
    I^{}_{11}  \leq &\ \epsilon^{-2} \norm{F'(U^{}_h)}^{}_{L^{}_{\infty}(\Omega)} \norm{\rho}^{}_{L^{}_2(\Omega)}\norm{\rho}^{}_{L^{}_4(\Omega)} \norm{\int_{t}^{\uptau}  \rho^2(s) \diff s}^{}_{L^{}_4(\Omega)}\\
     \leq &\ \frac{ C^{1/4}_{PF} \tilde{c}}{\epsilon^{2}}   \norm{F'(U^{}_h)}^{}_{L^{}_{\infty}(\Omega)} \norm{\rho}^{}_{L^{}_2(\Omega)}\norm{\rho}^{}_{L^{}_4(\Omega)} \norm{\int_{t}^{\uptau} \nabla  \rho^2(s) \diff s}^{}_{L^{}_2(\Omega)} \\
    & \leq  \frac{1}{2}  \norm{ \rho}^2_{L^{}_2(\Omega)} +  \frac{1}{8} \norm{ \rho}^4_{L^{}_4(\Omega)} +  \frac{ C^{}_{PF} \tilde{c}^4 }{ 2  \epsilon^{8}} \norm{F'(U^{}_h)}^4_{L^{}_{\infty}(\Omega)} \norm{\int_{t}^{\uptau} \nabla \rho^2(s) \diff s}^4_{L^{}_2(\Omega)}.
\end{align*}
The estimation of the remaining $I^{}_j$ on the right-hand of \eqref{4.12} are completely analogous to the two-dimensional case with the difference that one applies \eqref{eq:GNL} for $d=3$.  Collecting all the estimates, we arrive at the desirable result.
\end{proof}

\begin{remark}
In the \emph{a posteriori} error estimation literature for evolution problems, $\mathcal{L}^{}_1$ and the term $\epsilon^{-4} \norm{F(U^{}_h) - F(U^n_h)}^2_{L^{}_2(\Omega)}$ of $\mathcal{L}_2$ are often referred to as the \textit{time error estimates}, while $\norm{f^n - f}^2_{L^{}_2(\Omega)} $ is the \textit{data approximation}.  $\Theta^{}_1$ represents the \textit{mesh change} and $\Theta^{}_2$ (or $\tilde{\Theta}^{}_2$, respectively) is often termed as the \textit{spatial error estimate}. These will be  presented in detail in Section \ref{sec:fully_computable}.
\end{remark}

\begin{remark}
	We stress that the above result remains valid for the case of Neumann boundary conditions, upon modifying slightly the definition of the elliptic reconstruction \eqref{4.3} to eliminate the undetermined mode. Moreover, this can be done in such a way to recover \eqref{eq:GNL} for terms involving $\rho$. This is not done here in the interest of simplicity of the presentation only.
	\end{remark}


\subsection{Spectral estimates} To ensure polynomial dependence of the resulting estimates on $\epsilon^{-1}$, a widely used idea is to employ \emph{spectral estimates} of the principal eigenvalue of the linearized Allen-Cahn operator:
\begin{equation}\label{spectral_classical}
-\lambda^{}(t) : = {\displaystyle\inf_{v\in H^{1}_0(\Omega) \setminus \lbrace{0 \rbrace}}} \frac{\norm{\nabla v}^2_{L_2(\Omega)} + \epsilon^{-2} \left( F'(u) v,v \right)}{\norm{v}^2_{L_2(\Omega)}}.
\end{equation}
The celebrated works \cite{Chen1994,Mottoni-Schatzman1995,Alikakos-Fusco1993}  showed that  $\lambda$ can be bounded independently of $\epsilon$ for the case of smooth, evolved interfaces. This idea was used in the seminal works \cite{Feng-Prohl2003} for the proof of \emph{a priori} and \cite{Kessler-Nochetto-Schmidt2004,Feng-Wu2005} for \emph{a posteriori} error bounds for finite element methods in various norms with constants depending upon $\epsilon^{-1}$ only in a polynomial fashion. The \emph{a priori} nature of the spectral estimate \eqref{spectral_classical} is somewhat at odds, however, with the presence of $\lambda$ in \emph{a posteriori} error bounds. This difficulty was overcome in the seminal work \cite{Bartels2005} by first linearizing  about the numerical solution $U_h$, viz.,
		\begin{align}\label{spectral_apost}
-\lambda_{h}(t) := \inf_{v\in H^{1}_0(\Omega) \setminus \lbrace{0 \rbrace}} \frac{\norm{\nabla v}^2_{L_2(\Omega)} + \epsilon^{-2} \left( F'(U^{}_h) v,v \right)}{\norm{v}^2_{L_2(\Omega)}},
\end{align}
and by then proving verifiable eigenvalue approximation error bounds. The latter ensure that it is possible to compute principle eigenvalue approximations $\Lambda_h>0$, such that $\Lambda_h \ge \lambda_h$; we refer to \cite[Section 5]{Bartels2005} for the detailed construction. In short, it has been shown that for linear conforming finite element spaces, ($\kappa=1$,) it is possible to construct $\Lambda_h(t)\ge \lambda_h(t)$ for almost all $t\in(0,T]$ upon assuming that $\norm{U_h}_{L_{\infty}(\Omega)}$ remains bounded independently of $\epsilon^{-1}$.

The $\epsilon$-independence $\lambda$, (resp.~$\lambda_h$, $\Lambda_h$,) however, is \emph{not} guaranteed when the evolving interfaces are subjected to topological changes. This is an important challenge, since phase-field approaches are preferred over sharp-interface models exactly due to their ability evolve interfaces past topological changes. To address this, in \cite{Bartels-Muller-Ortner2011} (cf., also  \cite{Bartels2016,Bartels-Muller2011}) a crucial observation on the temporal integrability of $\lambda$ under topological changes was given: during topological changes we have $\lambda\sim \epsilon^{-2}$, but \emph{only} for time periods of length $\epsilon^2$. Therefore, it has been postulated that there exists an $m > 0$,  such that 
\begin{equation}\label{top_change}
 \int_{0}^{T} (\lambda(t))_+ \diff{t} \leq C + \log{(\epsilon^{-m})}
 \end{equation}
holds  for some constant $C>0$ \emph{independent} of $\epsilon$, for some $m\ge 0$; notice that for $m=0$, we return to the earlier case of no topological changes. A number of numerically validated scenarios justifying \eqref{top_change} for the scalar Allen-Cahn and its vectorial counterpart, the Ginzburg-Landau equation, can be found in \cite{Bartels-Muller-Ortner2011}. Moreover, a construction for a  $\Lambda_h\in L_1(0,T)$ such that 
\begin{equation}\label{assumption}
\int_{0}^{T} (\Lambda_h(t))_+ \diff{t}\ge \int_{0}^{T} (\lambda_h(t))_+ (t) \diff{t}, 
\end{equation}
has been provided in \cite[Proposition 3.8]{Bartels-Muller-Ortner2011}.

The above motivate the following assumption on the behaviour of the principal eigenvalue $\lambda_h$, which we shall henceofrth adopt.

\begin{assumption}\label{ass:spec_est} We postulate the validity of one of the following options:
	\begin{itemize}
		\item[(I)] we assume that the zero level set $\Gamma_t = \lbrace x \in \Omega : u(x,t) = 0 \rbrace$ is sufficiently smooth.  Then, for almost every $t\in(0,T]$, there exists a   computable bound $ \Lambda_h(t)\ge  \lambda_{h}(t) $ which is \emph{independent} of $\epsilon$.
		
		\item[(II)] there exists an $m > 0$,  such that $ \int_{0}^{T} \lambda_h(t) \diff{t} \leq C + \log{(\epsilon^{-m})}$ for some constant $C>0$ \emph{independent} of $\epsilon$ and we can construct a $ \Lambda_h\in L_1(0,T)$ such that \eqref{assumption} holds.
	\end{itemize}
	\end{assumption}

Of course, Assumption \ref{ass:spec_est}(I) is a special case of Assumption \ref{ass:spec_est}(II), arising when $m=0$. Nonetheless, when  Assumption \ref{ass:spec_est}(I) is valid, the resulting \emph{a posteriori} error estimates will have more favourable dependence on the final time $T$ than the estimates that are possible under the more general Assumption \ref{ass:spec_est}(II).

We shall prove \emph{a posteriori} error estimates under the more general Assumption \ref{ass:spec_est}(II), commenting, nevertheless, on the differences that would arise in the proof under \ref{ass:spec_est}(I) instead.


\subsection{Continuation argument}\label{sec:cont_arg}

We begin by noting that, compared to the state-of-the-art estimates of \cite{Bartels-Muller-Ortner2011,Bartels-Muller2011}, there are three additional terms on the right hand side of \eqref{4.10}, \eqref{d=3}, due to the use of the special test function \eqref{4.11}:  $\|\theta\|_{L^{}_4(0,T;L^{}_4(\Omega))}$ and $\|\theta_t\|_{L^{}_4(0,T;L^{}_4(\Omega))}$ which arise naturally and are symmetric
	with respect to the $\|.\|_{L^{}_4(0,T;L^{}_4(\Omega))}$ norm that is to be estimated, while the additional term $\|.\|_{L^{}_6(0,T;L^{}_6(\Omega))}$ can be compensated by the presence of the additional terms $A(t)$ (weighted norms) appearing on the left-hand side. Since the $L^{}_6(0,T;L^{}_6(\Omega))$-norm does not arise naturally in the Allen-Cahn energy functions, we have opted in dropping the $A(t)$ terms in the analysis below. 

Assuming that $\Lambda_h$ is available, we set $v = \rho \in H_0^1(\Omega) $ in \eqref{spectral_apost}, to deduce
\begin{equation}\label{4.15}
\begin{aligned}
    &  \norm{\nabla \rho}^2_{L^{}_2(\Omega)}  + \epsilon^{-2} \left( F'(U_h^{})  \rho, \rho \right)   \\
     \geq & - {\Lambda}_{h}(t) (1 - \epsilon^2) \norm{ \rho}^2_{L^{}_2(\Omega)} + \epsilon^2 \norm{\nabla \rho}^2_{L^{}_2(\Omega)}  + \left( F'(U_h^{})  \rho, \rho \right).
\end{aligned}
\end{equation}
For $d=2$, we work as follows. Upon setting 
\begin{equation*}\label{4.16}
\eta_2:=\bigg( \frac{1}{2}  \norm{\rho(0)}^2_{L^{}_2(\Omega)}  + \frac{C_{PF}^2}{2} \norm{\rho(0)}^4_{L_4^{}(\Omega)} + \sum_{n=1}^{N} \int_{J^{}_n} \big( \Theta^{}_1 +  \Theta^{}_2 +  C_0( \mathcal{L}^{}_1 + \mathcal{L}^{}_2 )\big) \diff t\bigg)^{\!1/4},
\end{equation*}
$
\mathcal{D}^{}_2:=  \max\{4,\alpha(U_h)+ 2{\Lambda}_{h}(t) (1 - \epsilon^2) + 2\},$ and $
\mathcal{B}^{}_2:=  \max\{ 16 \beta(\theta,U_h),\gamma(\theta,U_h) \},$
we use \eqref{4.15}  on the left-hand side of \eqref{4.10}, we note that  $-F'(U_h)\le 1$, and ignore $\int_{0}^{\uptau} A(t) \diff t$, to arrive at

\begin{align*}
	&\quad \frac{1}{4} \int_{0}^{\uptau} \norm{\rho}^4_{L^{}_4(\Omega)} \diff t +\frac{\epsilon^2}{2} \int_{0}^{\uptau} \norm{\nabla \rho}^2_{L^{}_2(\Omega)}  +  \frac{1}{8} \norm{\int_{0}^{\uptau}  \nabla \rho^2(s) \diff s}^2_{L^{}_2(\Omega)}  +  \frac{1}{2}  \norm{\rho(\uptau)}^2_{L^{}_2(\Omega)}\\
	\leq &\quad \eta_2^4    + \int_{0}^{\uptau} \mathcal{D}^{}_2(t)\Big( \frac{1}{8} \norm{\int_{0}^{\uptau} \nabla \rho^2(s) \diff s}^2_{L_2^{}(\Omega)} +\frac{1}{2}\norm{\rho}^2_{L^{}_2(\Omega)}\Big)  \diff t \\
	& +\epsilon^{-6} \int_{0}^{\uptau}  \mathcal{B}^{}_2(t)\Big( \frac{1}{64} \norm{\int_{t}^{\uptau} \nabla \rho^2(s) \diff s}^4_{L_2^{}(\Omega)} +\frac{1}{4}\norm{\rho}^4_{L^{}_2(\Omega)}\Big) \diff t\\
	\leq  & \quad \eta^4_{2}  +  \int_{0}^{T} \mathcal{D}^{}_2(t) \left( \frac{1}{8}   \norm{\int_{0}^{\uptau} \nabla \rho^2(s) \diff s}^2_{L_2^{}(\Omega)} + \frac{1}{2}  \norm{\rho}^2_{L^{}_2(\Omega)} \right) \diff t \\
	& + \frac{ \Bar{\mathcal{B}}^{}_2}{\epsilon^{6}} \sup_{t\in[0,\uptau]} \Big\{\frac{1}{8} \ \norm{\int_{0}^{\uptau}\! \nabla \rho^2(s) \diff s}^2_{L_2^{}(\Omega)}  +  \frac{1}{2}  \norm{\rho}^2_{L^{}_2(\Omega)} \Big\} \\
	& \qquad\times \bigg(   \frac{\uptau}{8}  \norm{\int_{0}^{\uptau} \!\nabla \rho^2(s) \diff s}^2_{L_2^{}(\Omega)}  +    \sup_{t\in[0,\uptau]}  \frac{\uptau}{2}  \norm{\rho}^2_{L^{}_2(\Omega)}   \bigg).
\end{align*}
where  $\Bar{\mathcal{B}}^{}_2 : = \sup_{t\in[0,T]} \mathcal{B}^{}_2(t)$.

Now, we set $ E^{}_2: = \exp \left( \int_{0}^{T} \mathcal{D}^{}_2(t) \diff{t} \right)$ and, for $d=2,3$, we use the abbreviation
\[
\begin{aligned}
\mathcal{N}_{[0,\uptau],d}(\rho):=&\ \frac{1}{4(d-1)}  \int_{0}^{\uptau}  \norm{\rho}^4_{L_4^{}(\Omega)} \diff t    + \frac{\epsilon^2}{2}  \int_{0}^{\uptau} \norm{\nabla \rho}^2_{L^{}_2(\Omega)} \diff t\\ &+\frac{1}{8}  \norm{\int_{0}^{\uptau} \nabla \rho^2(s) \diff s}^2_{L_2^{}(\Omega)} +  \sup_{t\in[0,\uptau]} \frac{1}{2}  \norm{\rho}^2_{L^{}_2(\Omega)},
\end{aligned}
\]
for the collection of semi-norms on the left-hand side of the last estimate. With this notation, we define the set
\begin{align*}
    I_2 : = \Big\lbrace \uptau\in[0,T] :  \mathcal{N}_{[0,\uptau],2}(\rho)
    \leq  4 \eta^4_2 E^{}_2   \Big\rbrace.
\end{align*}
The set $I^{}_2$ is non-empty because  $0\in I^{}_2$ and the left-hand side depends continuously on $\uptau$. We set $ \uptau^{*} := \max I^{}_2 $, and we assume that $ \uptau^{*} < T$; we aim to arrive at a contradiction. Hence, using the definition of  the set $I^{}_2$, we deduce 
\begin{align*}
    \mathcal{N}_{[0,\uptau],2}(\rho)
    \leq & \  \eta^4_2  +  \int_{0}^{\uptau} \mathcal{D}^{}_2(t) \Big( \frac{1}{8}   \norm{\int_{0}^{\uptau} \nabla \rho^2(s) \diff s}^2_{L_2^{}(\Omega)} + \frac{1}{2}  \norm{\rho}^2_{L^{}_2(\Omega)} \Big) \diff t \\
 &      +  16 \Bar{\mathcal{B}}^{}_2  \  \eta^8_2 E^{2}_2 ( T  +1)\epsilon^{-6}.
\end{align*}
If the last term on the right-hand side of the last estimate is bounded above by $ \eta^4_2$, or, equivalently, if it holds
\begin{equation}\label{4.18}
    \eta^4_2 \leq    \epsilon^{6}\big(16 \Bar{\mathcal{B}}^{}_2  (T + 1) E^{2}_2\big)^{-1},
\end{equation}
then for all $0 \leq \uptau \leq \uptau^{*} $ we  have
\begin{align*}
    \mathcal{N}_{[0,\uptau],2}(\rho)
    \leq & \ 2 \eta^4_2 +  \int_{0}^{\uptau} \mathcal{D}^{}_2(t) \Big( \frac{1}{8}   \norm{\int_{0}^{\uptau} \nabla \rho^2(s) \diff s}^2_{L_2^{}(\Omega)} + \frac{1}{2}  \norm{\rho}^2_{L^{}_2(\Omega)} \Big) \diff t.
\end{align*}
Since  $\frac{1}{8}   \norm{\int_{0}^{\uptau} \nabla \rho^2(s) \diff s}^2_{L_2^{}(\Omega)} + \frac{1}{2}  \norm{\rho}^2_{L^{}_2(\Omega)}\le  \mathcal{N}_{[0,\uptau],2}(\rho)$, Gr\"{o}nwall's Lemma implies 
\begin{align*}
   \mathcal{N}_{[0,\uptau^*],2}(\rho) \leq  2  \eta^4_2 E^{}_2,
\end{align*}
upon setting $\uptau = \uptau^{*}$. This contradicts the hypothesis $ \uptau^{*} < T $ and, therefore, proves that $ I^{}_2 = [0,T]$.

Likewise for $d=3$, we insert the spectral estimate \eqref{4.15} into \eqref{d=3}, and we work as for $d=2$. Setting 
\begin{equation*}\label{n_3}
    \eta_3:=\bigg(\frac{1}{2}  \norm{\rho(0)}^2_{L^{}_2(\Omega)}  + \frac{C_{PF}^2}{2} \norm{\rho(0)}^4_{L_4^{}(\Omega)} + \sum_{n=1}^{N} \int_{J^{}_n} \big( \Theta^{}_1 + \tilde{\Theta}^{}_2 +\tilde{C}_0( \mathcal{L}^{}_1 + \mathcal{L}^{}_2 ) \big) \diff t\bigg)^{\!1/4},
\end{equation*}
$\mathcal{D}^{}_3:=  \max\{4,\alpha(U_h) + 2{\Lambda}_{h}(t) (1 - \epsilon^2) + 3 \},$ $
\mathcal{B}^{}_3:=  \max\{ 16 \tilde{\beta}(\theta,U_h),\tilde{\gamma}(\theta,U_h) \},$ and $\Bar{\mathcal{B}}^{}_3 : = \sup_{t\in[0,T]} \mathcal{B}^{}_3(t)$,
through the same argumentation, we conclude that now the set $  I_3 : = \lbrace \uptau\in[0,T] :  \mathcal{N}_{[0,\uptau],3}(\rho)
\leq  4 \eta^4_3 E^{}_3   \rbrace$ equals $[0,T]$ upon assuming the condition  
\begin{equation}\label{4.19}
  \eta^4_3 \leq    \epsilon^{10}\big(16 \Bar{\mathcal{B}}^{}_3  (T + 1) E^{2}_3\big)^{-1}.
\end{equation}

The above argument has already confirmed the validity of the following result.


\begin{lemma} Assume that \eqref{4.18} holds when  when $d=2$ or \eqref{4.19} holds when $d=3$. Then, we have the bound
\begin{equation}\label{4.20}
  \mathcal{N}_{[0,T],d}(\rho)\le   4  \eta^4_{d} E^{}_{d}.
\end{equation}
\end{lemma}

\subsection{Main results}
Now we are ready to present the main error estimate  in the $L^{}_4(0,T;L^{}_4(\Omega))$-norm, from which we can easily arrived at a fully computable \emph{a posteriori} estimate in Section \ref{sec:fully_computable}. 
\begin{theorem}\label{main_theorem}
Let  $u^{}_0 \in L^{}_{\infty}(\Omega)$ and $f\in L^{}_{\infty}(0,T;L^{}_4(\Omega))$, $\Omega \subset \mathbb{R}^d$, $d= 2,3$. Let $u$ be the solution of \eqref{2.1} and $U^{}_h$ is its approximation \eqref{3.2}. Then, under Assumption \eqref{ass:spec_est}{\rm (II)} and the condition
\begin{equation}\label{conditional_estimates}
\eta^{}_d  \leq    \big(16 (T + 1) \Bar{\mathcal{B}}^{}_d E_d^2\big)^{-1/4} \epsilon^{d-1/2} ,
\end{equation}
the following error bound holds
\begin{align}\label{4.21}
    \norm{u-U^{}_h}^{}_{L^{}_4(0,T;L^{}_4( \Omega ))} \leq  2 \eta^{}_d \left( (d-1)  E^{}_{d} \right)^{1/4} + \norm{\theta}^{}_{L^{}_4(0,T;L^{}_4( \Omega ))}.
\end{align}
\end{theorem}
\begin{proof}
Ignoring nonnegative terms on the left-hand side of \eqref{4.20}, we have
\begin{align*}
    \norm{\rho}^{}_{L^{}_4(0,T;L^{}_4( \Omega ))} \leq 2 \eta^{}_{d} \left( (d-1) E^{}_{d} \right)^{1/4};
\end{align*}
 the proof follows by a triangle inequality.
\end{proof}

\begin{remark}
	Under the more restrictive Assumption \ref{ass:spec_est}(I), the continuation argument presented in Section \ref{sec:cont_arg} remains analogous with minor alterations. Specifically, if we set $m=0$ and we replace ${\Lambda}_{h}(t)$ by $\lambda$ and  $E^{}_d = \exp \left( \int_{0}^{T} \mathcal{D}^{}_{d}(t) \diff(t) \right)$ by  $E^{}_d = \exp \left( \bar{\mathcal{D}}^{}_{d} T \right) $,  with $\bar{\mathcal{D}}^{}_{d} := \sup_{t\in[0,T]} 	\max\{4,\alpha(U_h) + 2\lambda (1 - \epsilon^2) + d \},$  $d=2,3$, 
	Theorem \ref{main_theorem} remains valid.
\end{remark}

\begin{remark}
	 We stress that Theorem \ref{main_theorem} holds also in cases whereby it is \emph{not} possible to assume that $\norm{U^{}_h}_{L^{}_{\infty}(0,T;L_{\infty}(\Omega))}$ is bounded independently of $\epsilon$. We note, however, that $\norm{U^{}_h}_{L^{}_{\infty}(0,T;L_{\infty}(\Omega))}$ remains uniformly bounded with respect to $\epsilon$ and the mesh parameters in all scenarios of practical interest we are aware of and it is typically required in scenarios ensuring the validity of Assumption \ref{ass:spec_est}.
	\end{remark}

It is instructive to discuss in detail the dependence of the various terms appearing in \eqref{conditional_estimates} and \eqref{4.21} to assess the practicality of the resulting \emph{a posteriori} error bound below. The computational challenge for $\epsilon\ll 1$ is manifested by the satisfaction of the condition \eqref{conditional_estimates}. Indeed as $\epsilon\to 0$ the condition \eqref{conditional_estimates} becomes increasingly more stringent to be satisfied, necessitating meshes to be increasingly locally fine enough so as to reduce the estimator $\eta_d$; this results to proliferation of the numerical degrees of freedom. Once $\eta_d$ is small enough, an adaptive algorithm could make use of Theorem \ref{main_theorem} for further estimation, which requires  \eqref{conditional_estimates} to be valid.

Assume for argument's sake that $\norm{U_h^n}^{}_{L_{\infty}(\Omega)} \leq C'$ for all $n=1,\ldots, N$ for some $\epsilon$-independent constant $C' > 0$.  Also, we have
\[
 \norm{\theta}^{}_{L^{}_{\infty}(0,T;L^{}_{\infty}(\Omega))} 
=  \norm{\ell^{}_{n-1}(t) \theta^{n-1} + \ell^{}_{n}(t) \theta^{n}}^{}_{L^{}_{\infty}(0,T;L^{}_{\infty}(\Omega))}
\leq \max_{n=1, \ldots, N} \norm{\theta^n}^{}_{L^{}_{\infty}(\Omega)}.
\]
The $L^{}_{\infty}(\Omega)$-norm of each $\theta^n$ will be further estimated  in Section \ref{sec:fully_computable}. For the moment, if also assume that $\norm{\theta^n}^{}_{L_{\infty}(\Omega)} \leq C'$ uniformly with respect to $\epsilon$, then we can conclude that $6\le \bar{\mathcal{B}}_d\le CC'$, $d=2,3$ and, therefore,
	\[
	3\le  2 \big((T + 1) \Bar{\mathcal{B}}^{}_d\big)^{1/4} \le C(T+1)^{1/4},
	\]
	for some generic constants $C>0$, independent of $\epsilon$, upon noting that $\sqrt[4]{6}>1.5$.

Moreover, 
	in the case of smooth developed interfaces (Assumption \ref{ass:spec_est}(I)), one expects that $E^{}_d \sim 1$ as highlighted in the classical works \cite{Chen1994, Mottoni-Schatzman1995}. When topological changes take place, we can follow \cite{Bartels-Muller-Ortner2011} and postulate that $ E^{}_d \sim \epsilon^{-m}$, $m>0$. With the above convention, we find that \eqref{conditional_estimates} becomes
	\begin{align*}
	\eta^{}_d & \leq  G_d  \epsilon^{d+(m-1)/2},
	\end{align*}
	for some constant $G_d\ge1$ for all $m\ge 0$, thus encapsulating simultaneously both cases of Assumption \ref{ass:spec_est}.

	Hence, the $\epsilon$-dependence for the condition \eqref{conditional_estimates} appears to be \emph{less} stringent than in the respective conditional \emph{a posteriori} in $L_{\infty}(L_2)$- and $L_{2}(H^1)$-norms from \cite{Bartels2005,Bartels-Muller-Ortner2011,Bartels-Muller2011}, which reads, roughly speaking, $\tilde{\eta}\le c \epsilon^{4+3m}$ for the corresponding estimator $\tilde{\eta}$ and some constant $c>0$. Therefore, seeking to prove \emph{a posteriori} error estimates for the $L_4(L_4)$-norm error is, in our view, justified, as they can be potentially used to drive space-time adaptive algorithms without excessive numerical degree of freedom proliferation. This is an significant undertaking in its own right and will be considered in detail elsewhere.
	
	The new \emph{a posteriori} error analysis appears to also improve the $\epsilon$-dependence on the condition for  $L^{}_2(H^1)$- and $L_\infty^{}(L_2)$-norm bounds compared to \cite{Feng-Wu2005, Bartels2005,Bartels-Muller-Ortner2011,Bartels-Muller2011} in certain cases. Of course, the different method of proof above results to different terms appearing in $\eta_d$ above compared to the respective conditional \emph{a posteriori} error bounds from \cite{Feng-Wu2005,Bartels2005,Bartels-Muller-Ortner2011,Bartels-Muller2011}. Therefore, the performance of the proposed estimates above has to be assessed numerically before any conclusive statements can made. In particular, we have the following result.
	\begin{proposition}[$L^{}_2(H^1)$- and $L^{}_{\infty}(L^{}_2)$-norm estimates] With the hypotheses of Theorem \ref{main_theorem} and, assuming condition \eqref{conditional_estimates},
		we have the bounds
		\begin{align*}
		\norm{u-U^{}_h}^{}_{L^{}_2(0,T;H^1_0(\Omega))} & \leq  2 \sqrt{2}  \epsilon^{-1} \eta^{2}_d  E_{d}^{1/2} + \norm{\theta}^{}_{L^{}_2(0,T;H^1_0(\Omega))},\\
		\norm{u-U^{}_h}^{}_{L^{}_{\infty}(0,T;L^{}_2(\Omega))} & \leq  2 \sqrt{2} \eta^{2}_d  E_{d}^{1/2} + \norm{\theta}^{}_{L^{}_{\infty}(0,T;L^{}_2(\Omega))}.
		\end{align*}
	\end{proposition}
	Therefore, in the same setting as before, we have \eqref{conditional_estimates} implies
	\begin{align*}
	\eta^{2}_{d} \leq  G^2_d \epsilon^{2d-1 + m}.
	\end{align*}
	If we accept that $\eta^{2}_{d}\sim \tilde{\eta}$ from \cite{Bartels2005,Bartels-Muller-Ortner2011,Bartels-Muller2011}, for the sake of the argument, at least at the level of the conditional estimate, \eqref{conditional_estimates} gives formally favourable dependence on $\epsilon$ when $d=2$ and $m\ge 0$ and also when $d=3$ and $m\ge 1/2$, compared to the respective dependence 
	$\tilde{\eta}\le c \epsilon^{4+3m}$ from \cite{Bartels-Muller-Ortner2011,Bartels-Muller2011}.


\section{Fully computable upper bound}\label{sec:fully_computable}
The bound in Theorem \ref{main_theorem} is still not fully computable, due various terms involving $\theta^{}$  and $\rho(0)$, which we shall now further estimate by computable quantities.

\subsection{Initial condition estimates}
For the terms involving $\rho(0)$, we have
\begin{align*}
    \frac{1}{2}  \norm{\rho(0)}^2_{L^{}_2(\Omega)} & \leq   \norm{u^{}_0 - U^0_h}^2_{L^{}_2(\Omega)} + \norm{\theta^0_{}}^2_{L^{}_2(\Omega)}, \\
    \frac{C_{PF}^2}{2}  \norm{\rho(0)}^4_{L^{}_4(\Omega)} & \leq 4 C_{PF}^2 \big( \norm{u_0^{} - U_h^0}^4_{L_4^{}(\Omega)}  +  \norm{\theta^0_{}}^4_{L^{}_4(\Omega)} \big).
\end{align*}

 The Sobolev norms of $\theta$ appearing on $\eta_d$ can be further estimated by \emph{a posteriori} bounds for elliptic problems; see, e.g., \cite{Verfurth1996,Ainsworth-Oden2000}. We focus, therefore, in the derivation of $L_p$-norm \emph{a posteriori} error bounds for elliptic problems for $\theta$ and for $\theta_t$  via suitable duality arguments. Although the derivation is somewhat standard, we prefer to present it here with some level of detail to highlight the regularity assumptions required. Specifically, consider the dual problem: 
\begin{equation}\label{5.1}
    - \Delta z   = {\psi}^{p-1}_{}  \quad \text{in}  \ \Omega , \qquad
    z  = 0   \quad  \text{on}  \ \partial\Omega; 
\end{equation}
on an $\Omega \subset \mathbb{R}^d$ convex domain. Then, there exists a constant $C^{}_{\Omega} > 0$, depending on the domain $\Omega$, such that
\begin{align}\label{5.2}
    \norm{z}^{}_{W^{2,{p}/p-1}(\Omega)} & \leq C^{}_{\Omega} \norm{\psi^{p-1}}^{}_{L^{}_{{p}/p-1}(\Omega)} = C^{}_{\Omega} \norm{\psi}^{p-1}_{L^{}_{p}(\Omega)}, \ \text{for} \ p\ge 2;
\end{align}
we refer to \cite{Grisvard2011} for details. 


\subsection{Spatial error estimates} We shall estimate $\Theta^{}_2$ by residual-type estimators due to the presence of non-Hilbertian norms. In view of Remark \ref{remark_GO} above,  $\theta^n=w^n-U_h^n$ is the error of the elliptic problem \eqref{4.3}, so we can further estimate norms of $\theta$ once we have estimators of the form 
\[
 \norm{\theta^n}^{}_{L^{}_p(\Omega)}  \le  \mathcal{E}(U_h^{n},g_h^{n}; L^{}_p(\Omega)), 
\]
at our disposal for $p=2,4,6$. Therefore, from \eqref{4.5} we have
\begin{align*}
    \norm{\theta}^{}_{L^{}_p(\Omega)}  
    & \leq \mathcal{E}\left(U_h^{n},g_h^{n}; L^{}_p(\Omega) \right) + \mathcal{E}\left(U_h^{n-1},g_h^{n-1}; L^{}_p(\Omega) \right),
\end{align*}giving
\begin{equation}
\begin{aligned}
    \sum_{n=1}^{N} \int_{J^{}_n} \norm{\theta}^{p}_{L^{}_p(\Omega)} \diff t \leq \hat{c} \sum_{n=1}^{N} k^{}_n \left( \mathcal{E}^{p} \left(U_h^{n},g_h^{n}; L^{}_p(\Omega) \right) + \mathcal{E}^{p} \left(U_h^{n-1},g_h^{n-1}; L^{}_p(\Omega) \right)  \right), 
\end{aligned}
\end{equation}
for $\hat{c} > 0$ an algebraic constant.

Let $2\le p<+\infty$. To determine the estimator $\mathcal{E}$ precisely, we set $\psi=\theta^n$ on \eqref{5.1} and we have
\[
     \norm{\theta^{n}}^{p}_{L_p^{}(\Omega)} 
      = \int_{\Omega}  \nabla z {\cdot}  \nabla \theta^{n} \diff x - \int_{\Omega} \nabla \mathcal{I}^{n}_h z  {\cdot}  \nabla \theta^{n} \diff x
    = \int_{\Omega}  \nabla \left( z - \mathcal{I}^{n}_h z \right) {\cdot}  \nabla \left( \omega^{n} - U_h^{n}\right) \diff x 
    \]
     from Remark \ref{remark_GO}, with  $\mathcal{I}^{n}_h : W^{1,1}(\Omega), \rightarrow  V^{n}_{hk}$ denoting the standard Scott-Zhang interpolation operator that satisfies optimal approximation properties \cite{Scott-Zang1990}. Continuing in standard fashion, we have
    \begin{align*}
    \norm{\theta^{n}}^{p}_{L_p^{}(\Omega)}  =&\  \sum_{\tau \in \mathcal{T}^{n}_h} \int_{\tau} \nabla \omega^{n}  {\cdot} \nabla \left( z - \mathcal{I}^{n}_h z \right)   \diff x + \sum_{\tau \in \mathcal{T}^{n}_h} \int_{\tau} \Delta U_h^{n}  {\cdot} \left( z - \mathcal{I}^{n}_h z \right) \diff x\\
    & \ \quad - \sum_{\tau \in \mathcal{T}^{n}_h} \int_{\partial \tau} \left( \nabla U_h^{n} \cdot \vec{n} \right) \left(  z - \mathcal{I}^{n}_h z \right) \diff s  \\
     =& \sum_{\tau \in \mathcal{T}^{n}_h} \int_{\tau}r_n\left(  z - \mathcal{I}^{n}_h z \right)  \diff x -  \int_{\mathcal{S}_h^{n}} \llbracket \nabla U^{n}_h \rrbracket \left(  z - \mathcal{I}^{n}_h z \right) \diff s \\
     \leq& C_{SZ}\Big( \! \sum_{\tau \in \mathcal{T}^{n}_h} \norm{ h_n^2r_n }^{p}_{L^{}_{p}(\tau)}   \!+\!  \sum_{\varepsilon \in \mathcal{S}_h^{n}}  \norm{h^{1+1/p}_{n}\llbracket \nabla U^{n}_h \rrbracket}^{p}_{L^{}_{p}(\varepsilon)}\! \Big)^{1/p}\!\!\norm{z}^{}_{W^{2,\frac{p}{p-1}}(\Omega)},
\end{align*}
with $r_n:=\mathit{g}^{n}_{h}+ \Delta U_h^{n}$,
for some constant $C_{SZ}>0$ independent of $h_n$ and of the functions involves, using the approximation properties of $\mathcal{I}_h^n$;  here $\llbracket \nabla U^{n}_h \rrbracket$ is the \textit{jump} across the internal edge $\varepsilon$.
Then, the elliptic regularity estimate \eqref{5.2} implies that
\begin{equation*}\label{5.8}
    \mathcal{E}\left(U_{h}^{n},g_{h}^{n};L^{}_{p}(\Omega) \right) : = \ C^{}_{\Omega}C_{SZ} \Big( \! \sum_{\tau \in \mathcal{T}^{n}_h} \norm{ h_n^2r_n }^{p}_{L^{}_{p}(\tau)}   \!+\!  \sum_{\varepsilon \in \mathcal{S}_h^{n}}  \norm{h^{1+1/p}_{n}\llbracket \nabla U^{n}_h \rrbracket}^{p}_{L^{}_{p}(\varepsilon)}\! \Big)^{1/p},
\end{equation*}
the \textit{element residual} at time $t^{}_n$. 

For the limiting case $p=+\infty$,  we can take 
\begin{equation*}\label{pointwise-estimator}
\begin{aligned}
\mathcal{E}\left(U_{h}^{n},g_{h}^{n};L^{}_{\infty}(\Omega) \right) := \ C \ell^{}_{h,d}\Big( \! \sum_{\tau \in \mathcal{T}^{n}_h} \norm{ h_n^2r_n }_{L^{}_{\infty}(\tau)}   \!+\!  \sum_{\varepsilon \in \mathcal{S}_h^{n}}  \norm{h_{n}\llbracket \nabla U^{n}_h \rrbracket}_{L^{}_{\infty}(\varepsilon)}\! \Big),
\end{aligned}
\end{equation*}
with $\ell^{}_{h,d} = \left( \ln{(1/h^{}_n)} \right)^{\alpha^{}_d}$, where $\alpha^{}_2 = 2$ and $\alpha^{}_3=1$; we refer to \cite{Demlow-Georgoulis2012} for details.


\subsection{Mesh change estimates} 
The general strategy of time extensions in \eqref{4.5}, \eqref{4.6} consists in decomposing   $\theta^{}_t$ as follows
\begin{equation}\label{5.4}
    \theta^{}_t = \omega^{}_t - U^{}_{h,t} = \frac{\omega^n - \omega^{n-1}}{k^{}_n} - \frac{U^{n}_h - U^{n-1}_h}{k^{}_n}, \quad \text{for each} \  n=1,\ldots,N,
\end{equation}
with $U_h^{m} \in V^m$, $m=1,\ldots,N$. Since $V^n_{hk} \neq V^{n-1}_{hk}$in general, we define the Scott-Zang interpolation operator $\hat{\mathcal{I}}_h^{n}: H^1_0(\Omega) \rightarrow  V^n_{hk} \cap V^{n-1}_{hk}$ relative to the \textit{finest common coarsening} $\hat{\mathcal{T}}^{n}_h$ of $\mathcal{T}^{n}_h$ and $\mathcal{T}_{h}^{n-1}$. The latter allows to apply the Galerkin orthogonality property of the  \textit{elliptic reconstruction} in $V^n_{hk} \cap V^{n-1}_{hk}$. Moreover, we have the following approximation result: for all $\varepsilon \in  \check{\mathcal{S}}_h^{n} \backslash \hat{\mathcal{S}}_h^{n}$ and $1 \leq p < \infty$ it holds that
\begin{equation}\label{scott-zang}
    \norm{\psi - \hat{\mathcal{I}}_h^{n} \psi}^{}_{L^{}_{p}(\varepsilon)} \leq C^{}_{SZ} (\max_{\omega(\varepsilon)}\hat{h}_n)^{l-1/p} \norm{\psi}^{}_{W^{l,p}(\omega(\varepsilon))} \quad  \ \forall  l \leq \kappa + 1,
\end{equation}
where $\hat{h}^{}_n := \max(h^{}_n,h^{}_{n-1})$, with $\omega(\varepsilon)$ denoting the neighbourhood of elements sharing the face $\varepsilon$, where, as before, the positive constant $C^{}_{SZ}$ depends only on the shape regularity of the triangulation.

Setting $\psi=\theta_t$ on \eqref{5.1}, we work as before to deduce
\begin{align*}
 \norm{\theta^{}_t}^{p}_{L^{}_2(\Omega)}  
     =&\  k^{-1}_n \int_{\Omega} \nabla \big( z - \hat{\mathcal{I}}^n_h z \big) {\cdot} \nabla \left( \omega^n - \omega^{n-1} - U^{n}_h + U^{n-1}_h \right) \diff x \\
     =&\sum_{\tau \in \check{\mathcal{T}}^n_h} \int_{\tau} \partial \mathit{r}^{}_n \big(  z - \hat{\mathcal{I}}^n_h z \big) \diff x - \sum_{\varepsilon \in \check{\mathcal{S}}_h^{n}} \int_{\varepsilon} \partial \llbracket \nabla U^{n}_h \rrbracket \big(  z - \hat{\mathcal{I}}^n_h z \big) \diff s,
\end{align*}
with $ \check{\mathcal{S}}_h^{n}$ denoting the finite element space subordinate to the coarsest common refinement $\check{\mathcal{T}}_h^n$ of $\mathcal{T}_h^{n}$ and $\mathcal{T}_h^{n-1}$ and $\partial X_n:=(X_n-X_{n-1})/k_n$ for some sequence $\{X_n\}_{n\in \mathbb{N}\cup\{0\}}$. Standard estimation via H\"{o}lder's inequality and  \eqref{scott-zang} give, in turn,
 \begin{align*}
   \norm{\theta^{}_t}^{p}_{L^{}_{p}(\Omega)}      
      \leq &  C^{}_{SZ} \bigg( \sum_{\tau \in \check{\mathcal{T}}^n_h} \norm{ \hat{h}^{2}_n\partial \mathit{r}^{}_n}^{p}_{L^{}_{p}(\tau)}   +  \sum_{\varepsilon \in \check{\mathcal{S}}_h^{n}}  \norm{\hat{h}^{1+1/p}_n\partial \llbracket \nabla U^{n}_h \rrbracket}^{p}_{L^{}_{p}(\varepsilon)} \bigg)^{1/p}\norm{  z }^{}_{W^{2,\frac{p}{p-1}}(\Omega)}.
\end{align*}
Finally, the assumed elliptic regularity \eqref{5.2}, gives the \emph{a posteriori} error estimator
\begin{equation*}\label{5.5}
\begin{aligned}
    \hat{\mathcal{E}}\left(U_{h,t}^{},g_{h,t}^{};L^{}_{p}(\Omega) \right) := & \   C_{\Omega}C^{}_{SZ} \bigg( \sum_{\tau \in \check{\mathcal{T}}^n_h} \norm{ \hat{h}^{2}_n\partial \mathit{r}^{}_n}^{p}_{L^{}_{p}(\tau)}   +  \sum_{\varepsilon \in \check{\mathcal{S}}_h^{n}}  \norm{\hat{h}^{1+1/p}_n\partial \llbracket \nabla U^{n}_h \rrbracket}^{p}_{L^{}_{p}(\varepsilon)} \bigg)^{1/p},
\end{aligned}
\end{equation*}
for which we have $ \norm{\theta^{}_t}^{p}_{L^{}_{p}(\Omega)} \le  \hat{\mathcal{E}}^p(U_{h,t}^{},g_{h,t}^{};L^{}_{p}(\Omega) )$.

\section*{Aknowledgments}
EHG acknowledges the financial support of The Leverhulme Trust via a research project grant (grant no. RPG-2015-306). DP acknowledges the financial support of the Stavros Niarchos Foundation within the framework of project ARCHERS.

\bibliographystyle{siam}
\bibliography{bib_allen-cahn}

\end{document}